\documentclass[12pt]{amsart}
\usepackage{amssymb}
\usepackage{graphicx}
\newtheorem{thm}{Theorem}
\newtheorem{cor}{Corollary}
\newtheorem{lemma}{Lemma}

\newcommand{\mat}[4]{\mbox{\small 
$\left[\!\addtolength{\arraycolsep}{-1pt}\begin{array}{cc} #1&#2\\[-1pt]
#3&#4\end{array}\!\right]$}}
\newcommand{\vect}[2]{\mbox{\small 
$\left[\!\addtolength{\arraycolsep}{-1pt}\begin{array}{cc} #1\\[-1pt]
#2\end{array}\!\right]$}}
\begin{document}\raggedbottom
\title[Modular forms]{\mbox{ }\\[-40pt] Modular forms, projective structures, 
and the four squares theorem}
\author[M. Eastwood]{Michael Eastwood}
\address{\hskip-\parindent
School of Mathematical Sciences\\
University of Adelaide\\ 
SA 5005\\ 
Australia}
\email{meastwoo@gmail.com}
\author[B. Moore]{Ben Moore}
\address{\hskip-\parindent
Mathematics Institute\\
Zeeman Building\\
University of Warwick\newline
Coventry CV4 7AL\\
England}
\email{benmoore196884@gmail.com}
\begin{abstract} 
It is well-known that Lagrange's four-square theorem, stating that every
natural number may be written as the sum of four squares, may be proved using
methods from the classical\linebreak theory of modular forms and theta
functions. We revisit this proof. In doing so, we concentrate on geometry and
thereby avoid some of the tricky analysis that is often encountered. Guided by
projective differential geometry we find a new route to Lagrange's theorem.
\end{abstract}
\renewcommand{\subjclassname}{\textup{2010} Mathematics Subject Classification}
\maketitle
\begin{center}
\begin{picture}(340,340)
\put(170,170){\makebox(0,0){
\includegraphics[width=\linewidth]{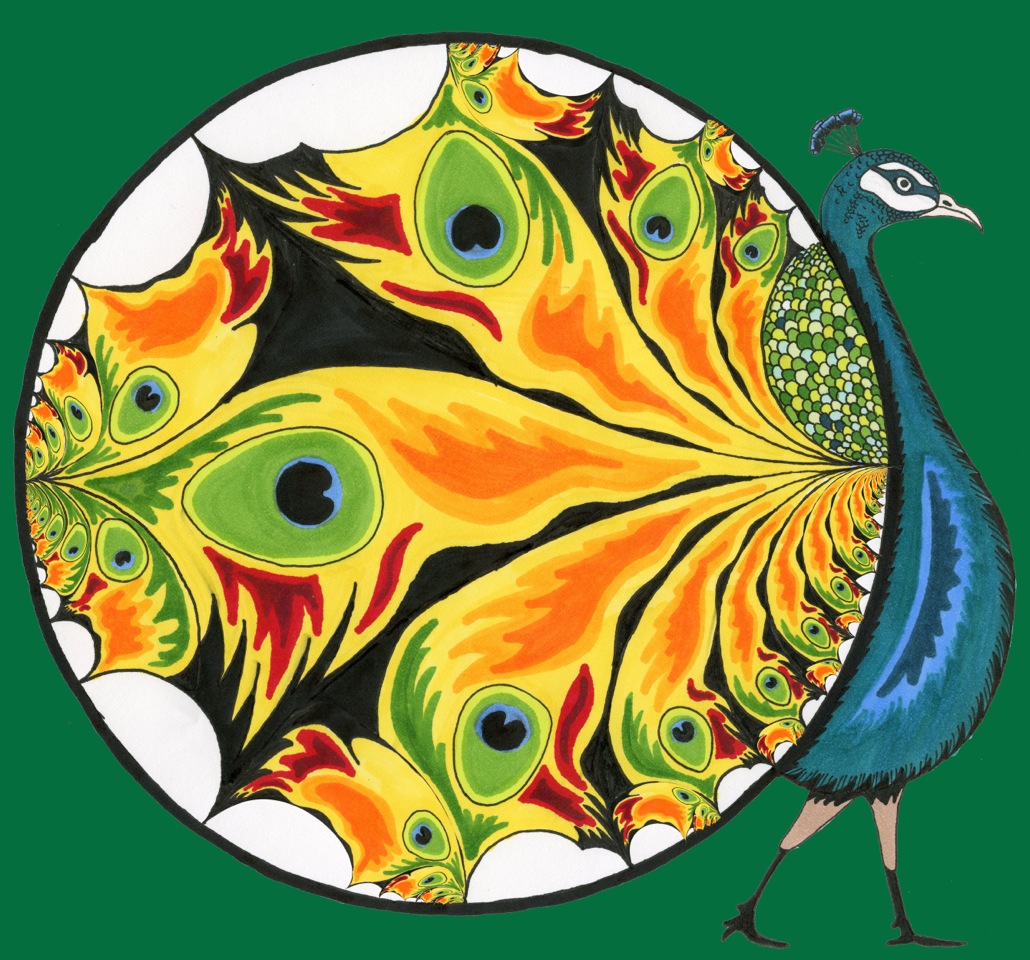}}}
\end{picture}\end{center}
\begin{center}
An artist's impression of the action of $\Gamma_0(4)$ on the unit 
disc.\\
In the Polish wycinanka {\l}owicka style by Katarzyna Nurowska.
\end{center}
\section{Introduction}
In 1770, Lagrange proved 
that every natural number can be 
written as the sum of four squares. In 1834, Jacobi gave a formula for the 
number of different ways that this can be done. More precisely, if we consider 
the formal power series
\begin{equation}\label{theta}\theta(q)\equiv
\sum_{n\in{\mathbb{Z}}}q^{n^2}=1+2(q+q^4+q^9+q^{16}+q^{25}+\cdots),
\end{equation}
then Lagrange's Theorem says that all coefficients of 
$$(\theta(q))^4=1+8(q+3q^2+4q^3+3q^4+6q^5+12q^6+8q^7+3q^8+13q^9+\cdots)$$
are positive  whilst Jacobi's Theorem gives a manifestly positive formula
for these coefficients. 
In fact, it is evident from the identity
$$2(a^2+b^2+c^2+d^2)=(a+b)^2+(a-b)^2+(c+d)^2+(c-d)^2,$$
that, for Lagrange's theorem, it suffices to show that all odd natural numbers
may be written as the sum of four squares whence it suffices to establish 
Jacobi's formula in this case, namely that
\begin{equation}\label{jacobi}\begin{array}{rcl}(\theta(q))^4-(\theta(-q))^4
&\!\!\!=\!\!\!&16(q+4q^3+6q^5+8q^7+13q^9+\cdots)\\[4pt]
&\!\!\!=\!\!\!&16(\sum_{m=0}^\infty\sigma(2m+1)q^{2m+1}),
\end{array}\end{equation}
where $\sigma(n)\equiv\sum_{d|n}d$ is the {\em sum-of-divisors function\/}. The
aim of this article is to prove~(\ref{jacobi}). It is well-known that this can
be accomplished using modular forms and this is what we shall do. However, some
of the tricky analysis can be avoided in favour of geometry. This is one
motivation for this article. Another is that a key feature of the usual proof,
namely that a certain vector space of modular forms is two-dimensional, is
replaced by the two-dimensionality of the solution space to a projectively
invariant linear differential equation. This reasoning is potentially
applicable for automorphic forms beyond complex analysis.

\section{The twice-punctured sphere}
It is not commonly realised that the first contributor to the theory of modular
forms was the cartographer Mercator, who in 1569 found an accurate conformal
map of the twice-punctured round sphere. With the punctures at the South and
North Poles, this {\em Mercator projection\/} is the default representation of
the Earth to be found in ordinary atlases\footnote{But we find it
convenient to put the southern hemisphere at the top.}. {From} a modern
perspective, it may be constructed in two steps:\pagebreak
\begin{itemize}
\item Use stereographic projection

{\setlength{\unitlength}{.7cm}
\begin{picture}(5,4.8)(-8,-2.9)
\linethickness{0.3mm}
 \qbezier(1.9,0.0)(1.9,0.787)(1.3435,1.3435)
 \qbezier(1.3435,1.3435)(0.787,1.9)(0.0,1.9)
 \qbezier(0.0,1.9)(-0.787,1.9)(-1.3435,1.3435)
 \qbezier(-1.3435,1.3435)(-1.9,0.787)(-1.9,0.0)
 \qbezier(-1.9,0.0)(-1.9,-0.787)(-1.3435,-1.3435)
 \qbezier(-1.3435,-1.3435)(-0.787,-1.9)(0.0,-1.9)
 \qbezier(0.0,-1.9)(0.787,-1.9)(1.3435,-1.3435)
 \qbezier(1.3435,-1.3435)(1.9,-0.787)(1.9,0.0)
 \qbezier(-1.9,0.0)(-1.9,-0.4)(0.0,-0.4)
 \qbezier(0.0,-0.4)(1.9,-0.4)(1.9,0.0)
{\linethickness{0.2mm}\qbezier(1.9,0.0)(1.9,0.4)(0.0,0.4)
 \qbezier(0.0,0.4)(-1.9,0.4)(-1.9,0.0)}
\thicklines
\put(-3,-2.7){\line(1,0){7}}
\put(-3,-2.7){\line(-1,1){1}}
\put(4,-2.7){\line(-1,1){1}}
\put(.2,2){$N$}
\put(0,1.9){\makebox(0,0){\large$\bullet$}}
\put(0,1.9){\line(1,-2){2.1}}
\put(2.1,-2.3){\makebox(0,0){\large$\bullet$}}
\put(1.35,-0.8){\makebox(0,0){\large$\bullet$}}
\put(0,-1.9){\makebox(0,0){\large$\bullet$}}
\put(.2,-1.7){$S$}
\end{picture}}

\noindent to identify $S^2\setminus\{N\}$ with the complex
plane~${\mathbb{C}}$. 
\item Use the complex logarithm to `unwrap' the punctured complex plane 
${\mathbb{C}}\setminus\{0\}$ to its universal cover~${\mathbb{C}}$. 
\end{itemize}
These two steps are conformal, the first by geometry or calculus, and 
the second by the Cauchy-Riemann equations. Explicit formul{\ae} are
$$\begin{array}{ccccc}{\mathbb{C}}&\longrightarrow
&{\mathbb{C}}\setminus\{0\}&\longrightarrow&S^2\setminus\{S,N\}\\[5pt]
\tau&\longmapsto&q=e^{2\pi i\tau}\\[-7pt]
&&q=u+iv&\longmapsto&\mbox{\footnotesize$\displaystyle\frac{1}{u^2+v^2+4}
\left[\begin{array}{c}4u\\ 4v\\ u^2+v^2-4\end{array}\right]$}
\end{array}$$
and we end up with two crucial (and conformal) facts:
\begin{itemize}
\item $S^2\setminus\{S,N\}\cong
\displaystyle\frac{\mathbb{C}}{\{\tau\sim\tau+1\}}$,
\item $q=e^{2\pi i\tau}$ is a local co\"ordinate on~$S^2$ near the South Pole.
\end{itemize}
Note that this essential appearance of the logarithm in the Mercator 
projection predates Napier and others (in 
the seventeenth century).

The Mercator realisation of the twice-punctured sphere
$$S^2\setminus\{S,N\}=S^2\setminus\{q=0,q=\infty\}$$
may already be used to prove some useful identities as follows.
\begin{thm} If\/ $q=e^{2\pi i\tau}$, then 
\begin{equation}\label{rite}
\sum_{d=-\infty}^{\infty}\frac1{(\tau+d)^2}=-4\pi^2\sum_{m=1}^\infty mq^m,
\quad\mbox{for}\enskip |q|<1.\end{equation}
\end{thm}
\begin{proof} It is easy to check that the left hand side is uniformly
convergent on compact subsets of~${\mathbb{C}}\setminus{\mathbb{Z}}$. It is
invariant under $\tau\mapsto\tau+1$ and therefore descends to a
holomorphic function on the thrice-punctured sphere:
$$S^2\setminus\{q=0,q=\infty,q=1\}.$$
Let us call this function $F(q)$ and note that
\begin{itemize}
\item $F(q)\to 0$ as $q\to 0$,
\item $F(1/q)=F(q)$.
\end{itemize}
It follows that $F(q)$ extends holomorphically through $q=0$ and $q=\infty$ and
has zeroes at these two points whilst at $q=1$ it clearly extends
meromorphically with a double pole there. Hence,
$$F(q)=C\frac{q}{(q-1)^2}$$
for some constant~$C$. To compute $C$, we may substitute $\tau=1/2$ to find that
$$C=-16\sum_{d=-\infty}^\infty\frac1{(2d+1)^2}=-16\frac{\pi^2}4=-4\pi^2.$$
Finally, if $|q|<1$, then
$$\frac{q}{(q-1)^2}=q\frac\partial{\partial q}\frac1{1-q}
=q\frac\partial{\partial q}\sum_{m=0}^\infty q^m=\sum_{m=1}^\infty mq^m,$$
as required.
\end{proof}
\begin{cor} For $q=e^{2\pi i\tau}$ and $|q|<1$,
\begin{equation}\label{another_rite}
\sum_{d=-\infty}^{\infty}\frac1{(\tau+d)^4}
=\frac{8\pi^4}3\sum_{m=1}^\infty m^3q^m.\end{equation}
\end{cor}
\begin{proof} By the chain rule
$$\frac\partial{\partial\tau}=2\pi i q\frac\partial{\partial q},$$
and applying this operator twice to~(\ref{rite}) gives the required identity.  
\end{proof}
We remark that identities such as (\ref{rite}) and (\ref{another_rite}) are
often established using `unfamiliar expressions' for trigonometric functions
and regarded as a `standard rite of passage into modular
forms'~\cite[p.~5]{DS}. Already, we see the utility of the Mercator projection
in identifying the universal cover of the twice-punctured sphere and it is
natural to ask about a similar identification for the thrice-punctured sphere.

\section{The thrice-punctured sphere}\label{thrice_punctured}
Our exposition in this section follows advice from Tony Scholl to the 
the first author in 1984\@.

Let $\Sigma$ be the thrice-punctured Riemann sphere. More specifically, let us
use the standard co\"ordinate
$z\in{\mathbb{C}}\hookrightarrow{\mathbb{C}}\sqcup\{\infty\}=S^2$, and set
$$\Sigma\equiv S^2\setminus\{0,1,\infty\}
=\{z\in{\mathbb{C}}\mid z\not=0,1\}.$$
By the Riemann mapping theorem there is a conformal isomorphism between the
lower half plane
$$\{z=x+iy\in{\mathbb{C}}\mid y<0\}$$
and the following subset
\begin{equation}\label{ideal_UHP}
\raisebox{-32pt}{\setlength{\unitlength}{.4cm}
\begin{picture}(5,6)(-6,-1)
\linethickness{0.3mm}
 \qbezier(1.9,0.0)(1.9,0.787)(1.3435,1.3435)
 \qbezier(1.3435,1.3435)(0.787,1.9)(0.0,1.9)
 \qbezier(0.0,1.9)(-0.787,1.9)(-1.3435,1.3435)
 \qbezier(-1.3435,1.3435)(-1.9,0.787)(-1.9,0.0)
\thicklines
\put(-3.8,0){\vector(1,0){10}}
\put(-1.9,0){\vector(0,1){6}}
\put(1.9,0){\line(0,1){6}}
\put(-1.9,0){\makebox(0,0){$\bullet$}}
\put(1.9,0){\makebox(0,0){$\bullet$}}
\put(-1.9,-.7){\makebox(0,0){$\tau\!=\!0$}}
\put(1.9,-.7){\makebox(0,0){$\tau\!=\!1/2$}}
\put(5,.3){$s$}
\put(-1.6,4.7){$t$}
\thinlines
\put(-1.9,.8){\line(1,0){.2}}
\put(1.9,.8){\line(-1,0){.2}}
\put(-1.9,1.1){\line(1,0){.3}}
\put(1.9,1.1){\line(-1,0){.3}}
\put(-1.9,1.4){\line(1,0){.6}}
\put(1.9,1.4){\line(-1,0){.6}}
\put(-1.9,1.7){\line(1,0){1}}
\put(1.9,1.7){\line(-1,0){1}}
\put(-1.9,2){\line(1,0){3.8}}
\put(-1.9,2.3){\line(1,0){3.8}}
\put(-1.9,2.6){\line(1,0){3.8}}
\put(-1.9,2.9){\line(1,0){3.8}}
\put(-1.9,3.2){\line(1,0){3.8}}
\put(-1.9,3.5){\line(1,0){3.8}}
\put(-1.9,3.8){\line(1,0){3.8}}
\put(-1.9,4.1){\line(1,0){3.8}}
\put(-1.9,4.4){\line(1,0){3.8}}
\put(1.9,4.7){\line(-1,0){3}}
\put(1.9,5){\line(-1,0){2.8}}
\put(1.9,5.3){\line(-1,0){2.8}}
\put(1.9,5.6){\line(-1,0){3}}
\put(1.9,5.9){\line(-1,0){3.2}}
\end{picture}}
\end{equation}
of the upper half plane ${\mathcal{H}}\equiv\{\tau=s+it\mid t>0\}$. In fact, 
as with all Riemann mappings, there is a three-parameter family thereof 
and we need to specify just one of them. To do this let us extend the lower 
half plane as the complement of two rays
\begin{center}\begin{picture}(0,45)
\thicklines
\put(-30,25){\vector(-1,0){130}}
\put(30,25){\vector(1,0){130}}
\put(30,25){\makebox(0,0){$\bullet$}}
\put(-30,25){\makebox(0,0){$\bullet$}}
\put(-30,15){\makebox(0,0){$z=0$}}
\put(30,15){\makebox(0,0){$z=1$}}
\put(0,25){\makebox(0,0){\Large$\bullet$}}
\thinlines
\put(0,25){\vector(1,0){20}}
\end{picture}\end{center}
extend the target domain as
\begin{center}
{\setlength{\unitlength}{.4cm}
\begin{picture}(5,6)(-6,-1)
\linethickness{0.3mm}
 \qbezier(1.9,0.0)(1.9,0.787)(1.3435,1.3435)
 \qbezier(1.3435,1.3435)(0.787,1.9)(0.0,1.9)
 \qbezier(0.0,1.9)(-0.787,1.9)(-1.3435,1.3435)
 \qbezier(-1.3435,1.3435)(-1.9,0.787)(-1.9,0.0)
 \qbezier(5.7,0.0)(5.7,0.787)(5.1435,1.3435)
 \qbezier(5.1435,1.3435)(4.587,1.9)(3.8,1.9)
 \qbezier(3.8,1.9)(3.013,1.9)(2.4565,1.3435)
 \qbezier(2.4565,1.3435)(1.9,0.787)(1.9,0.0)
\thicklines
\put(-3.8,0){\line(1,0){10}}
\put(-1.9,0){\line(0,1){6}}
\put(5.7,0){\line(0,1){6}}
\put(-1.9,0){\makebox(0,0){$\bullet$}}
\put(1.9,0){\makebox(0,0){$\bullet$}}
\put(5.7,0){\makebox(0,0){$\bullet$}}
\put(-1.9,-.7){\makebox(0,0){$\tau\!=\!0$}}
\put(1.9,-.7){\makebox(0,0){$\tau\!=\!1/2$}}
\put(5.7,-.7){\makebox(0,0){$\tau\!=\!1$}}
\thinlines
\put(-1.9,.8){\line(1,0){.2}}
\put(1.9,.8){\line(-1,0){.2}}
\put(1.9,.8){\line(1,0){.2}}
\put(5.7,.8){\line(-1,0){.2}}
\put(-1.9,1.1){\line(1,0){.3}}
\put(1.9,1.1){\line(-1,0){.3}}
\put(1.9,1.1){\line(1,0){.3}}
\put(5.7,1.1){\line(-1,0){.3}}
\put(-1.9,1.4){\line(1,0){.6}}
\put(1.9,1.4){\line(-1,0){.6}}
\put(1.9,1.4){\line(1,0){.6}}
\put(5.7,1.4){\line(-1,0){.6}}
\put(-1.9,1.7){\line(1,0){1}}
\put(1.9,1.7){\line(-1,0){1}}
\put(1.9,1.7){\line(1,0){1}}
\put(5.7,1.7){\line(-1,0){1}}
\put(-1.9,2){\line(1,0){7.6}}
\put(-1.9,2.3){\line(1,0){7.6}}
\put(-1.9,2.6){\line(1,0){7.6}}
\put(-1.9,2.9){\line(1,0){7.6}}
\put(-1.9,3.2){\line(1,0){7.6}}
\put(-1.9,3.5){\line(1,0){7.6}}
\put(-1.9,3.8){\line(1,0){7.6}}
\put(-1.9,4.1){\line(1,0){7.6}}
\put(-1.9,4.4){\line(1,0){7.6}}
\put(-1.9,4.7){\line(1,0){7.6}}
\put(-1.9,5){\line(1,0){7.6}}
\put(-1.9,5.3){\line(1,0){7.6}}
\put(-1.9,5.6){\line(1,0){7.6}}
\put(-1.9,5.9){\line(1,0){7.6}}
\put(1.9,3.8){\makebox(0,0){\Large$\bullet$}}
\put(1.9,3.8){\vector(0,-1){2}}
\end{picture}}
\end{center}
and consider the Riemann mapping between these extensions that sends
$z\!=\!1/2$ to $\tau\!=\!(1+i)/2$ and, at these points, sends the direction
$\partial/\partial x$ to $-\partial/\partial t$, as shown. 

This particular Riemann mapping is chosen so that it intertwines the involution
$z\mapsto 1-z$ (having fixed point $z\!=\!1/2$) with the involution
$\tau\mapsto(\tau-1)/(2\tau-1)$ of ${\mathcal{H}}$ (fixing $\tau\!=\!(1+i)/2$
and preserving the extended target). 

We conclude that the lower half plane is sent to the `tile'
\smallskip\begin{center}
{\setlength{\unitlength}{.4cm}
\begin{picture}(5,6)(-6,0)
\linethickness{0.3mm}
 \qbezier(1.9,0.0)(1.9,0.787)(1.3435,1.3435)
 \qbezier(1.3435,1.3435)(0.787,1.9)(0.0,1.9)
 \qbezier(0.0,1.9)(-0.787,1.9)(-1.3435,1.3435)
 \qbezier(-1.3435,1.3435)(-1.9,0.787)(-1.9,0.0)
\thicklines
\put(-3.8,0){\line(1,0){10}}
\put(-1.9,0){\line(0,1){6}}
\put(1.9,0){\line(0,1){6}}
\put(-1.9,0){\makebox(0,0){$\bullet$}}
\put(1.9,0){\makebox(0,0){$\bullet$}}
\thinlines
\put(-1.9,.8){\line(1,0){.2}}
\put(1.9,.8){\line(-1,0){.2}}
\put(-1.9,1.1){\line(1,0){.3}}
\put(1.9,1.1){\line(-1,0){.3}}
\put(-1.9,1.4){\line(1,0){.6}}
\put(1.9,1.4){\line(-1,0){.6}}
\put(-1.9,1.7){\line(1,0){1}}
\put(1.9,1.7){\line(-1,0){1}}
\put(-1.9,2){\line(1,0){3.8}}
\put(-1.9,2.3){\line(1,0){3.8}}
\put(-1.9,2.6){\line(1,0){3.8}}
\put(-1.9,2.9){\line(1,0){3.8}}
\put(-1.9,3.2){\line(1,0){3.8}}
\put(-1.9,3.5){\line(1,0){3.8}}
\put(-1.9,3.8){\line(1,0){3.8}}
\put(-1.9,4.1){\line(1,0){3.8}}
\put(-1.9,4.4){\line(1,0){3.8}}
\put(-3.6,1.8){\makebox(0,0)[r]{\framebox{$z=\infty\vphantom{1}$}}}
\put(-3.6,5){\makebox(0,0)[r]{\framebox{$z=0$}}}
\put(3.6,1.8){\makebox(0,0)[l]{\framebox{$z=1$}}}
\put(-3.6,1.7){\vector(1,-1){1.3}}
\put(3.6,1.7){\vector(-1,-1){1.3}}
\put(-3.6,5){\line(1,0){3.6}}
\put(0,5){\vector(0,1){1}}
\end{picture}}
\end{center}
\smallskip
and that this mapping holomorphically extends across the line segment $[0,1]$
to the upper half plane, which itself is sent to a neighbouring and translated
tile attached to the right of the original. It is illuminating to view this 
construction on the sphere
\begin{center}{\setlength{\unitlength}{1cm}
\begin{picture}(0,4)(0,-2)
\linethickness{0.3mm}
 \qbezier(1.9,0.0)(1.9,0.787)(1.3435,1.3435)
 \qbezier(1.3435,1.3435)(0.787,1.9)(0.0,1.9)
 \qbezier(0.0,1.9)(-0.787,1.9)(-1.3435,1.3435)
 \qbezier(-1.3435,1.3435)(-1.9,0.787)(-1.9,0.0)
 \qbezier(-1.9,0.0)(-1.9,-0.787)(-1.3435,-1.3435)
 \qbezier(-1.3435,-1.3435)(-0.787,-1.9)(0.0,-1.9)
 \qbezier(0.0,-1.9)(0.787,-1.9)(1.3435,-1.3435)
 \qbezier(1.3435,-1.3435)(1.9,-0.787)(1.9,0.0)
 \qbezier(-1.9,0.0)(-1.9,-0.4)(0.0,-0.4)
 \qbezier(0.0,-0.4)(1.9,-0.4)(1.9,0.0)
{\linethickness{0.2mm}\qbezier(1.9,0.0)(1.9,0.4)(0.0,0.4)
 \qbezier(0.0,0.4)(-1.9,0.4)(-1.9,0.0)}
\put(0,0.4){\makebox(0,0){\large$\bullet$}}
\put(-1.4,-.3){\makebox(0,0){\large$\bullet$}}
\put(1.4,-.3){\makebox(0,0){\large$\bullet$}}
\put(0,.72){\makebox(0,0){$\infty$}}
\put(-1.13,-0.05){\makebox(0,0){$0$}}
\put(1.13,-0.05){\makebox(0,0){$1$}}
\end{picture}}
\end{center}
with the lower hemisphere as domain and, replacing the upper half plane by the 
unit disc with its hyperbolic metric, the target is now the ideal 
triangle: 
\begin{equation}\label{ideal}
\raisebox{-53pt}{\setlength{\unitlength}{2cm}
\begin{picture}(2,2)(-1,-1)
\linethickness{0.3mm}
 \qbezier(1.000000,0.000000)(1.000000,0.267948)(0.866025,0.500000)
 \qbezier(0.866025,0.500000)(0.732052,0.732052)(0.500000,0.866025)
 \qbezier(0.500000,0.866025)(0.267948,1.000000)(0.000000,1.000000)
 \qbezier(0.000000,1.000000)(-0.267948,1.000000)(-0.500000,0.866025)
 \qbezier(-0.500000,0.866025)(-0.732052,0.732052)(-0.866025,0.500000)
 \qbezier(-0.866025,0.500000)(-1.000000,0.267948)(-1.000000,0.000000)
 \qbezier(-1.000000,0.000000)(-1.000000,-0.267948)(-0.866025,-0.500000)
 \qbezier(-0.866025,-0.500000)(-0.732052,-0.732052)(-0.500000,-0.866025)
 \qbezier(-0.500000,-0.866025)(-0.267948,-1.000000)(0.000000,-1.000000)
 \qbezier(0.000000,-1.000000)(0.267948,-1.000000)(0.500000,-0.866025)
 \qbezier(0.500000,-0.866025)(0.732052,-0.732052)(0.866025,-0.500000)
 \qbezier(0.866025,-0.500000)(1.000000,-0.267948)(1.000000,0.000000)
\put(1.000000,0.000000){\makebox(0,0){$\bullet$}}
\put(-0.500000,0.866025){\makebox(0,0){$\bullet$}}
\put(-0.500000,-0.866025){\makebox(0,0){$\bullet$}}
 \qbezier(1.000000,0.000000)(0.535898,0.000000)(0.133975,0.232051)
 \qbezier(1.000000,0.000000)(0.535898,0.000000)(0.133975,-0.232051)
 \qbezier(-0.500000,0.866025)(-0.267949,0.464102)(-0.267949,0.000000)
 \qbezier(-0.500000,0.866025)(-0.267949,0.464102)(0.133975,0.232051)
 \qbezier(-0.500000,-0.866025)(-0.267949,-0.464102)(0.133975,-0.232051)
 \qbezier(-0.500000,-0.866025)(-0.267949,-0.464102)(-0.267949,0.000000)
\put(1.5,0){\makebox(0,0)[l]{\framebox{$z=1$}}}
\put(-1.2,.9){\makebox(0,0)[r]{\framebox{$z=0$}}}
\put(-1.2,-.9){\makebox(0,0)[r]{\framebox{$z=\infty\vphantom{1}$}}}
\thinlines
\put(1.5,0){\vector(-1,0){.35}}
\put(-1.2,.9){\vector(1,0){.5}}
\put(-1.2,-.9){\vector(1,0){.5}}
\end{picture}}
\end{equation}
{From} this point of view, the mapping extends holomorphically through a
`portal' in the equator between $0$ and $1$ to the upper hemisphere, with the
result mapping to
\begin{center}{\setlength{\unitlength}{2cm}
\begin{picture}(2,2)(-1,-1)
\linethickness{0.3mm}
 \qbezier(1.000000,0.000000)(1.000000,0.267948)(0.866025,0.500000)
 \qbezier(0.866025,0.500000)(0.732052,0.732052)(0.500000,0.866025)
 \qbezier(0.500000,0.866025)(0.267948,1.000000)(0.000000,1.000000)
 \qbezier(0.000000,1.000000)(-0.267948,1.000000)(-0.500000,0.866025)
 \qbezier(-0.500000,0.866025)(-0.732052,0.732052)(-0.866025,0.500000)
 \qbezier(-0.866025,0.500000)(-1.000000,0.267948)(-1.000000,0.000000)
 \qbezier(-1.000000,0.000000)(-1.000000,-0.267948)(-0.866025,-0.500000)
 \qbezier(-0.866025,-0.500000)(-0.732052,-0.732052)(-0.500000,-0.866025)
 \qbezier(-0.500000,-0.866025)(-0.267948,-1.000000)(0.000000,-1.000000)
 \qbezier(0.000000,-1.000000)(0.267948,-1.000000)(0.500000,-0.866025)
 \qbezier(0.500000,-0.866025)(0.732052,-0.732052)(0.866025,-0.500000)
 \qbezier(0.866025,-0.500000)(1.000000,-0.267948)(1.000000,0.000000)
\put(1.000000,0.000000){\makebox(0,0){$\bullet$}}
\put(-0.500000,0.866025){\makebox(0,0){$\bullet$}}
\put(-0.500000,-0.866025){\makebox(0,0){$\bullet$}}
\put(0.500000,0.866025){\makebox(0,0){$\bullet$}}
\linethickness{0.2mm}
 \qbezier(1.000000,0.000000)(0.535898,0.000000)(0.133975,0.232051)
 \qbezier(-0.500000,0.866025)(-0.267949,0.464102)(0.133975,0.232051)
\linethickness{0.3mm}
 \qbezier(1.000000,0.000000)(0.535898,0.000000)(0.133975,-0.232051)
 \qbezier(-0.500000,0.866025)(-0.267949,0.464102)(-0.267949,0.000000)
 \qbezier(-0.500000,-0.866025)(-0.267949,-0.464102)(0.133975,-0.232051)
 \qbezier(-0.500000,-0.866025)(-0.267949,-0.464102)(-0.267949,0.000000)
 \qbezier(1.000000,0.000000)(0.666667,0.000000)(0.500000,0.288675)
 \qbezier(0.500000,0.866025)(0.333333,0.577350)(0.000000,0.577350)
 \qbezier(0.500000,0.866025)(0.333333,0.577350)(0.500000,0.288675)
 \qbezier(-0.500000,0.866025)(-0.333333,0.577350)(0.000000,0.577350) 
\put(1.5,0){\makebox(0,0)[l]{\framebox{$z=1$}}}
\put(1.5,.9){\makebox(0,0)[l]{\framebox{$z=\infty\vphantom{1}$}}}
\put(-1.2,.9){\makebox(0,0)[r]{\framebox{$z=0$}}}
\put(-1.2,-.9){\makebox(0,0)[r]{\framebox{$z=\infty\vphantom{1}$}}}
\thinlines
\put(1.5,0){\vector(-1,0){.35}}
\put(1.5,.9){\vector(-1,0){.8}}
\put(-1.2,.9){\vector(1,0){.5}}
\put(-1.2,-.9){\vector(1,0){.5}}
\put(-1.6,-.35){\makebox(0,0)[r]{\framebox{\begin{tabular}{r}Southern\\ 
Hemisphere\end{tabular}}}}
\put(-1.4,.35){\makebox(0,0)[r]{\framebox{\begin{tabular}{r}Northern\\ 
Hemisphere\end{tabular}}}}
\put(-1.4,.4){\vector(1,0){1.5}}
\put(-1.6,-.1){\vector(1,0){1.5}}
\end{picture}}
\end{center}
However, there are three such portals to the upper hemisphere, all on an equal
footing with respect to the evident three-fold rotational symmetry. Using all
three unwraps the thrice-punctured sphere to
\begin{equation}\label{tessellation}
\raisebox{-53pt}{\setlength{\unitlength}{2cm}
\begin{picture}(2,2)(-1,-1)
\linethickness{0.3mm}
 \qbezier(1.000000,0.000000)(1.000000,0.267948)(0.866025,0.500000)
 \qbezier(0.866025,0.500000)(0.732052,0.732052)(0.500000,0.866025)
 \qbezier(0.500000,0.866025)(0.267948,1.000000)(0.000000,1.000000)
 \qbezier(0.000000,1.000000)(-0.267948,1.000000)(-0.500000,0.866025)
 \qbezier(-0.500000,0.866025)(-0.732052,0.732052)(-0.866025,0.500000)
 \qbezier(-0.866025,0.500000)(-1.000000,0.267948)(-1.000000,0.000000)
 \qbezier(-1.000000,0.000000)(-1.000000,-0.267948)(-0.866025,-0.500000)
 \qbezier(-0.866025,-0.500000)(-0.732052,-0.732052)(-0.500000,-0.866025)
 \qbezier(-0.500000,-0.866025)(-0.267948,-1.000000)(0.000000,-1.000000)
 \qbezier(0.000000,-1.000000)(0.267948,-1.000000)(0.500000,-0.866025)
 \qbezier(0.500000,-0.866025)(0.732052,-0.732052)(0.866025,-0.500000)
 \qbezier(0.866025,-0.500000)(1.000000,-0.267948)(1.000000,0.000000)
 \qbezier(1.000000,0.000000)(0.535898,0.000000)(0.133975,0.232051)
 \qbezier(1.000000,0.000000)(0.535898,0.000000)(0.133975,-0.232051)
 \qbezier(-0.500000,0.866025)(-0.267949,0.464102)(-0.267949,0.000000)
 \qbezier(-0.500000,0.866025)(-0.267949,0.464102)(0.133975,0.232051)
 \qbezier(-0.500000,-0.866025)(-0.267949,-0.464102)(0.133975,-0.232051)
 \qbezier(-0.500000,-0.866025)(-0.267949,-0.464102)(-0.267949,0.000000)
\linethickness{0.2mm}
 \qbezier(1.000000,0.000000)(0.666667,0.000000)(0.500000,0.288675)
 \qbezier(1.000000,0.000000)(0.666667,0.000000)(0.500000,-0.288675)
 \qbezier(0.500000,0.866025)(0.333333,0.577350)(0.000000,0.577350)
 \qbezier(0.500000,0.866025)(0.333333,0.577350)(0.500000,0.288675)
 \qbezier(-0.500000,0.866025)(-0.333333,0.577350)(-0.500000,0.288675)
 \qbezier(-0.500000,0.866025)(-0.333333,0.577350)(0.000000,0.577350) 
 \qbezier(-1.000000,0.000000)(-0.666667,0.000000)(-0.500000,-0.288675)
 \qbezier(-1.000000,0.000000)(-0.666667,0.000000)(-0.500000,0.288675)
 \qbezier(-0.500000,-0.866025)(-0.33333,-0.577350)(0.000000,-0.577350)
 \qbezier(-0.500000,-0.866025)(-0.33333,-0.577350)(-0.500000,-0.288675)
 \qbezier(0.500000,-0.866025)(0.333333,-0.577350)(0.500000,-0.288675)
 \qbezier(0.500000,-0.866025)(0.333333,-0.577350)(0.000000,-0.577350)
\end{picture}}
\end{equation}
and, of course, we can keep going to and fro between north and south through
our three portals to obtain a tessellation\footnote{Familiar from the works of
M.C.~Escher.} of the hyperbolic disc~$\Delta$ and a conformal covering
$\Delta\to S^2\setminus\{0,1,\infty\}$. This is an explicit realisation of the
universal covering. We remark that the Little Picard Theorem follows
immediately from this realisation.

\section{Symmetries of the upper half plane}\label{upper_half_plane}
The reader may be wondering why we viewed the extended target as a domain in 
the upper half plane 
\begin{center}
{\setlength{\unitlength}{.4cm}
\begin{picture}(5,6)(-6,-1)
\linethickness{0.3mm}
 \qbezier(1.9,0.0)(1.9,0.787)(1.3435,1.3435)
 \qbezier(1.3435,1.3435)(0.787,1.9)(0.0,1.9)
 \qbezier(0.0,1.9)(-0.787,1.9)(-1.3435,1.3435)
 \qbezier(-1.3435,1.3435)(-1.9,0.787)(-1.9,0.0)
 \qbezier(5.7,0.0)(5.7,0.787)(5.1435,1.3435)
 \qbezier(5.1435,1.3435)(4.587,1.9)(3.8,1.9)
 \qbezier(3.8,1.9)(3.013,1.9)(2.4565,1.3435)
 \qbezier(2.4565,1.3435)(1.9,0.787)(1.9,0.0)
\thicklines
\put(-3.8,0){\line(1,0){10}}
\put(-1.9,0){\line(0,1){6}}
\put(5.7,0){\line(0,1){6}}
\put(-1.9,0){\makebox(0,0){$\bullet$}}
\put(1.9,0){\makebox(0,0){$\bullet$}}
\put(5.7,0){\makebox(0,0){$\bullet$}}
\put(-1.9,-.7){\makebox(0,0){$\tau\!=\!0$}}
\put(1.9,-.7){\makebox(0,0){$\tau\!=\!1/2$}}
\put(5.7,-.7){\makebox(0,0){$\tau\!=\!1$}}
\thinlines
\put(-1.9,.8){\line(1,0){.2}}
\put(1.9,.8){\line(-1,0){.2}}
\put(1.9,.8){\line(1,0){.2}}
\put(5.7,.8){\line(-1,0){.2}}
\put(-1.9,1.1){\line(1,0){.3}}
\put(1.9,1.1){\line(-1,0){.3}}
\put(1.9,1.1){\line(1,0){.3}}
\put(5.7,1.1){\line(-1,0){.3}}
\put(-1.9,1.4){\line(1,0){.6}}
\put(1.9,1.4){\line(-1,0){.6}}
\put(1.9,1.4){\line(1,0){.6}}
\put(5.7,1.4){\line(-1,0){.6}}
\put(-1.9,1.7){\line(1,0){1}}
\put(1.9,1.7){\line(-1,0){1}}
\put(1.9,1.7){\line(1,0){1}}
\put(5.7,1.7){\line(-1,0){1}}
\put(-1.9,2){\line(1,0){7.6}}
\put(-1.9,2.3){\line(1,0){7.6}}
\put(-1.9,2.6){\line(1,0){7.6}}
\put(-1.9,2.9){\line(1,0){7.6}}
\put(-1.9,3.2){\line(1,0){7.6}}
\put(-1.9,3.5){\line(1,0){7.6}}
\put(-1.9,3.8){\line(1,0){7.6}}
\put(-1.9,4.1){\line(1,0){7.6}}
\put(-1.9,4.4){\line(1,0){7.6}}
\put(-1.9,4.7){\line(1,0){7.6}}
\put(-1.9,5){\line(1,0){7.6}}
\put(-1.9,5.3){\line(1,0){7.6}}
\put(-1.9,5.6){\line(1,0){7.6}}
\put(-1.9,5.9){\line(1,0){7.6}}
\put(-5.5,2.5){\Large${\mathcal{D}}={}$}
\end{picture}}
\end{center}
rather than the corresponding domain in the unit disc:
\begin{center}{\setlength{\unitlength}{2cm}
\begin{picture}(2,2)(-1,-1)
\linethickness{0.3mm}
 \qbezier(1.000000,0.000000)(1.000000,0.267948)(0.866025,0.500000)
 \qbezier(0.866025,0.500000)(0.732052,0.732052)(0.500000,0.866025)
 \qbezier(0.500000,0.866025)(0.267948,1.000000)(0.000000,1.000000)
 \qbezier(0.000000,1.000000)(-0.267948,1.000000)(-0.500000,0.866025)
 \qbezier(-0.500000,0.866025)(-0.732052,0.732052)(-0.866025,0.500000)
 \qbezier(-0.866025,0.500000)(-1.000000,0.267948)(-1.000000,0.000000)
 \qbezier(-1.000000,0.000000)(-1.000000,-0.267948)(-0.866025,-0.500000)
 \qbezier(-0.866025,-0.500000)(-0.732052,-0.732052)(-0.500000,-0.866025)
 \qbezier(-0.500000,-0.866025)(-0.267948,-1.000000)(0.000000,-1.000000)
 \qbezier(0.000000,-1.000000)(0.267948,-1.000000)(0.500000,-0.866025)
 \qbezier(0.500000,-0.866025)(0.732052,-0.732052)(0.866025,-0.500000)
 \qbezier(0.866025,-0.500000)(1.000000,-0.267948)(1.000000,0.000000)
\put(1.000000,0.000000){\makebox(0,0){$\bullet$}}
\put(-0.500000,0.866025){\makebox(0,0){$\bullet$}}
\put(-0.500000,-0.866025){\makebox(0,0){$\bullet$}}
\put(0.500000,0.866025){\makebox(0,0){$\bullet$}}
 \qbezier(1.000000,0.000000)(0.535898,0.000000)(0.133975,-0.232051)
 \qbezier(-0.500000,0.866025)(-0.267949,0.464102)(-0.267949,0.000000)
 \qbezier(-0.500000,-0.866025)(-0.267949,-0.464102)(0.133975,-0.232051)
 \qbezier(-0.500000,-0.866025)(-0.267949,-0.464102)(-0.267949,0.000000)
 \qbezier(1.000000,0.000000)(0.666667,0.000000)(0.500000,0.288675)
 \qbezier(0.500000,0.866025)(0.333333,0.577350)(0.000000,0.577350)
 \qbezier(0.500000,0.866025)(0.333333,0.577350)(0.500000,0.288675)
 \qbezier(-0.500000,0.866025)(-0.333333,0.577350)(0.000000,0.577350) 
\thinlines
\put(-.41,-.7){\line(1,0){.02}}
\put(-.39,-.65){\line(1,0){.04}}
\put(-.38,-.6){\line(1,0){.06}}
\put(-.36,-.55){\line(1,0){.09}}
\put(-.35,-.5){\line(1,0){.13}}
\put(-.33,-.45){\line(1,0){.16}}
\put(-.31,-.4){\line(1,0){.2}}
\put(-.3,-.35){\line(1,0){.26}}
\put(-.29,-.3){\line(1,0){.32}}
\put(-.28,-.25){\line(1,0){.38}}
\put(-.27,-.2){\line(1,0){.46}}
\put(-.27,-.15){\line(1,0){.55}}
\put(-.27,-.1){\line(1,0){.69}}
\put(-.27,-.05){\line(1,0){.84}}
\put(-.27,0){\line(1,0){1.13}}
\put(-.27,.05){\line(1,0){1.02}}
\put(-.27,.1){\line(1,0){.93}}
\put(-.27,.15){\line(1,0){.88}}
\put(-.28,.2){\line(1,0){.85}}
\put(-.28,.25){\line(1,0){.8}}
\put(-.29,.3){\line(1,0){.78}}
\put(-.3,.35){\line(1,0){.76}}
\put(-.31,.4){\line(1,0){.75}}
\put(-.32,.45){\line(1,0){.75}}
\put(-.34,.5){\line(1,0){.76}}
\put(-.35,.55){\line(1,0){.76}}
\put(-.37,.6){\line(1,0){.2}}
\put(-.39,.65){\line(1,0){.1}}
\put(-.42,.7){\line(1,0){.04}}
\put(.18,.6){\line(1,0){.23}}
\put(.3,.65){\line(1,0){.12}}
\put(.37,.7){\line(1,0){.06}}
\end{picture}}
\end{center}
The point is that the upper half plane is more congenial with regard to an 
explicit realisation of the symmetry group for which this extended tile is a 
fundamental domain. 
\begin{lemma}\label{one} The two transformations
$$T\tau\equiv\tau+1\qquad\mbox{and}\qquad U\tau\equiv\frac\tau{4\tau+1}$$
generate a group of biholomorphisms of the upper half plane ${\mathcal{H}}$, 
having ${\mathcal{D}}$ as fundamental domain.
\end{lemma}
\begin{proof} Regarding the ideal triangle (\ref{ideal}), the corresponding
tessellation (\ref{tessellation}) is evidently generated by the three
hyperbolic reflections in its sides. Viewed in the upper half
plane~(\ref{ideal_UHP}), these three reflections are
$$\Pi_1\tau=-\overline\tau,\qquad\Pi_2\tau=1-\overline\tau,\qquad
\Pi_3\tau=\frac{\overline\tau}{4\overline\tau-1}.$$
Therefore, the group we seek may be generated by
$\Pi_2\circ\Pi_1$, $\Pi_3\circ\Pi_1$, and $\Pi_3\circ\Pi_2$, namely
$$\tau\mapsto
\tau+1,\qquad\tau\mapsto
\frac{\tau}{4\tau+1},\qquad\tau\mapsto
\frac{\tau-1}{4\tau-3}.
$$
But these three transformations are $T$, $U$, and $UT^{-1}$. 
%
\end{proof}
It is useful to have an algebraic description of the group generated
by $T$ and~$U$. To this end, and also because we shall need some of this 
algebra for other purposes later on, we record some well-known properties of 
the following well-known group.

\subsection{The modular group} This is an alternative name for the group 
${\mathrm{SL}}(2,{\mathbb{Z}})$, of $2\times 2$ unit determinant matrices with 
integer entries. It is generated by 
$$S\equiv\mat{0}{-1}{1}{0}\quad\mbox{and}\quad T\equiv\mat{1}{1}{0}{1}.$$
Notice that 
$$S^2=-{\mathrm{Id}}=(ST)^3.$$
There is a normal subgroup 
$\{\pm{\mathrm{Id}}\}\lhd{\mathrm{SL}}(2,{\mathbb{Z}})$. The 
quotient group is denoted ${\mathrm{PSL}}(2,{\mathbb{Z}})$. It is generated by 
$S$ and $T$ subject to the relations $S^2={\mathrm{Id}}=(ST)^3$. The group 
${\mathrm{SL}}(2,{\mathbb{R}})$ acts on the upper half plane ${\mathcal{H}}$ 
according to
$$\mat{a}{b}{c}{d}\tau=\frac{a\tau+b}{c\tau+d}\,,$$
this action descending to a faithful action of
${\mathrm{PSL}}(2,{\mathbb{Z}})$. Indeed, this action identifies
${\mathrm{PSL}}(2,{\mathbb{R}})$ as the biholomorphisms
of~${\mathcal{H}}$. Having done this, the subgroup
${\mathrm{PSL}}(2,{\mathbb{Z}})$ acts properly discontinuously
on~${\mathcal{H}}$. It is easy to verify and well-known that
\begin{equation}\label{usual_modular_fun}
\raisebox{-110pt}{\setlength{\unitlength}{90pt}
\begin{picture}(3,2.15)(-1,-.1)
\thicklines\put(0,0){\makebox(0,0){$\bullet$}}
\put(-1,0){\vector(1,0){2}}
\put(0,0){\vector(0,1){2}}
\put(.5,0){\makebox(0,0){$\bullet$}}
\put(0,1){\makebox(0,0){$\bullet$}}
\put(0,-.08){\makebox(0,0){$\tau=0$}}
\put(.5,-.08){\makebox(0,0){$\tau=1/2$}}
\put(.18,.8660){\makebox(0,0){$\tau=i$}}
\put(.5,.8660){\makebox(0,0){$\bullet$}}
\put(1.03,.8660){\makebox(0,0){$\tau=1/2+\sqrt{3}i/2$}}
\linethickness{1pt}
\qbezier (0.5,2) (0.5,2) (0.5,0.8660254040)
\qbezier (0.5,0.8660254040) (0.2679491924,1) (0,1)
\qbezier (0,1) (-0.2679491924,1) (-0.5,0.8660254040)
\qbezier (-0.5,0.8660254040) (-0.5,2) (-0.5,2)
\thinlines
\put(-.5,1.89){\line(1,0){1}}
\put(-.5,1.86){\line(1,0){1}}
\put(-.5,1.83){\line(1,0){1}}
\put(-.5,1.8){\line(1,0){1}}
\put(-.5,1.77){\line(1,0){1}}
\put(-.5,1.74){\line(1,0){1}}
\put(-.5,1.71){\line(1,0){1}}
\put(-.5,1.68){\line(1,0){1}}
\put(-.5,1.65){\line(1,0){1}}
\put(-.5,1.62){\line(1,0){1}}
\put(-.5,1.59){\line(1,0){1}}
\put(-.5,1.56){\line(1,0){1}}
\put(-.5,1.53){\line(1,0){1}}
\put(-.5,1.5){\line(1,0){1}}
\put(-.5,1.47){\line(1,0){1}}
\put(-.5,1.44){\line(1,0){1}}
\put(-.5,1.41){\line(1,0){1}}
\put(-.5,1.38){\line(1,0){1}}
\put(-.5,1.35){\line(1,0){1}}
\put(-.5,1.32){\line(1,0){1}}
\put(-.5,1.29){\line(1,0){1}}
\put(-.5,1.26){\line(1,0){1}}
\put(-.5,1.23){\line(1,0){1}}
\put(-.5,1.2){\line(1,0){1}}
\put(-.5,1.17){\line(1,0){1}}
\put(-.5,1.14){\line(1,0){1}}
\put(-.5,1.11){\line(1,0){1}}
\put(-.5,1.08){\line(1,0){1}}
\put(-.5,1.05){\line(1,0){1}}
\put(-.5,1.02){\line(1,0){1}}
\put(-.5,.99){\line(1,0){.35}}
\put(-.5,.96){\line(1,0){.22}}
\put(-.5,.93){\line(1,0){.13}}
\put(-.5,.9){\line(1,0){.06}}
\put(.5,.99){\line(-1,0){.35}}
\put(.5,.96){\line(-1,0){.22}}
\put(.5,.93){\line(-1,0){.13}}
\put(.5,.9){\line(-1,0){.06}}
\end{picture}}
\end{equation}
is a fundamental domain for this action.

\subsection{Some congruence subgroups}
Let us consider the following two subgroups of 
${\mathrm{SL}}(2,{\mathbb{Z}})$.
\begin{itemize}
\item $\Gamma(4)\equiv
\left\{\mat{a}{b}{c}{d}\in{\mathrm{SL}}(2,{\mathbb{Z}})
\mid\mat{a}{b}{c}{d}\equiv\mat{1}{0}{0}{1}\bmod 4\right\}$.
\item $\Gamma_1(4)\equiv\rule{0pt}{21pt}
\left\{\mat{a}{b}{c}{d}\in{\mathrm{SL}}(2,{\mathbb{Z}})
\mid\mat{a}{b}{c}{d}\equiv\mat{1}{*}{0}{1}\bmod 4\right\}$.
\end{itemize}
It is clear that
$$\Gamma(4)\lhd{\mathrm{SL}}(2,{\mathbb{Z}})\twoheadrightarrow
{\mathrm{SL}}(2,{\mathbb{Z}}_4)$$
and easily verified that ${\mathrm{SL}}(2,{\mathbb{Z}}_4)$ has 48 elements. 
In particular, the subgroup $\Gamma(4)$ has index 48 in 
${\mathrm{SL}}(2,{\mathbb{Z}})$. Also the homomorphism
$$\Gamma_1(4)\ni\mat{a}{b}{c}{d}\longmapsto b \bmod 4\in{\mathbb{Z}}_4$$
shows that $\Gamma(4)\lhd\Gamma_1(4)$
of index~$4$. Therefore, whilst $\Gamma_1(4)$ is not a
normal subgroup of ${\mathrm{SL}}(2,{\mathbb{Z}})$, it has index 
$48/4=12$.

We may now achieve our goal of an algebraic description of the group 
generated by $T$ and~$U$.
\begin{lemma}
The subgroup of ${\mathrm{SL}}(2,{\mathbb{Z}})$ generated
by
$$\mat{1}{1}{0}{1}\quad\mbox{and}\quad\mat{1}{0}{4}{1}$$
is $\Gamma_1(4)$.
\end{lemma}
\begin{proof} We give a geometric proof by comparing fundamental domains. To
this end we note that  
\begin{equation}\label{unusual_modular_fun}
\raisebox{-110pt}{\setlength{\unitlength}{90pt}
\begin{picture}(3,2.15)(-1,-.1)
\thicklines\put(0,0){\makebox(0,0){$\bullet$}}
\put(-1,0){\vector(1,0){2.5}}
\put(0,0){\vector(0,1){2}}
\put(.5,0){\makebox(0,0){$\bullet$}}
\put(1,0){\makebox(0,0){$\bullet$}}
\put(0,-.08){\makebox(0,0){$\tau=0$}}
\put(.5,-.08){\makebox(0,0){$\tau=1/2$}}
\put(1,-.08){\makebox(0,0){$\tau=1$}}
\put(0,1){\makebox(0,0){$\bullet$}}
\put(.5,0.8660254040){\makebox(0,0){$\bullet$}}
\put(-.22,1){\makebox(0,0){$\tau=i$}}
\put(1,.8){\makebox(0,0){$\tau=1/2+\sqrt{3}i/2$}}
\linethickness{1pt}
\qbezier (1,2) (1,2) (1,1)
\qbezier (1,1) (0.7320508076,1) (0.5,0.8660254040)
\qbezier (0.5,0.8660254040) (0.2679491924,1) (0,1)
\qbezier (0,1) (0,2) (0,2)
\thinlines
\put(0,1.89){\line(1,0){1}}
\put(0,1.86){\line(1,0){1}}
\put(0,1.83){\line(1,0){1}}
\put(0,1.8){\line(1,0){1}}
\put(0,1.77){\line(1,0){1}}
\put(0,1.74){\line(1,0){1}}
\put(0,1.71){\line(1,0){1}}
\put(0,1.68){\line(1,0){1}}
\put(0,1.65){\line(1,0){1}}
\put(0,1.62){\line(1,0){1}}
\put(0,1.59){\line(1,0){1}}
\put(0,1.56){\line(1,0){1}}
\put(0,1.53){\line(1,0){1}}
\put(0,1.5){\line(1,0){1}}
\put(0,1.47){\line(1,0){1}}
\put(0,1.44){\line(1,0){1}}
\put(0,1.41){\line(1,0){1}}
\put(0,1.38){\line(1,0){1}}
\put(0,1.35){\line(1,0){1}}
\put(0,1.32){\line(1,0){1}}
\put(0,1.29){\line(1,0){1}}
\put(0,1.26){\line(1,0){1}}
\put(0,1.23){\line(1,0){1}}
\put(0,1.2){\line(1,0){1}}
\put(0,1.17){\line(1,0){1}}
\put(0,1.14){\line(1,0){1}}
\put(0,1.11){\line(1,0){1}}
\put(0,1.08){\line(1,0){1}}
\put(0,1.05){\line(1,0){1}}
\put(0,1.02){\line(1,0){1}}
\put(.5,.99){\line(-1,0){.35}}
\put(.5,.96){\line(-1,0){.22}}
\put(.5,.93){\line(-1,0){.13}}
\put(.5,.9){\line(-1,0){.06}}
\put(.5,.99){\line(1,0){.35}}
\put(.5,.96){\line(1,0){.22}}
\put(.5,.93){\line(1,0){.13}}
\put(.5,.9){\line(1,0){.06}}
\end{picture}}\end{equation}
is a perfectly good alternative to the usual (\ref{usual_modular_fun}) as a
fundamental domain for the action of ${\mathrm{PSL}}(2,{\mathbb{Z}})$.
Moreover, six hyperbolic copies of this alternative may be used to tile the
fundamental domain ${\mathcal{D}}$ concerning the action of Lemma~\ref{one}:
\begin{equation}\label{look_six_copies}
\raisebox{-110pt}{\setlength{\unitlength}{90pt}
\begin{picture}(3,2.15)(-1,-.1)
\thicklines\put(0,0){\makebox(0,0){$\bullet$}}
\put(-1,0){\vector(1,0){2.5}}
\put(0,0){\vector(0,1){2}}
\put(.5,0){\makebox(0,0){$\bullet$}}
\put(1,0){\makebox(0,0){$\bullet$}}
\put(0,-.08){\makebox(0,0){$\tau=0$}}
\put(.5,-.08){\makebox(0,0){$\tau=1/2$}}
\put(1,-.08){\makebox(0,0){$\tau=1$}}
\linethickness{1pt}
\qbezier (1,2) (1,2) (1,0)
\qbezier (1,1) (0.7320508076,1) (0.5,0.8660254040)
\qbezier (0.5,0.8660254040) (0.2679491924,1) (0,1)
\qbezier (0,0) (0,2) (0,2)
\qbezier (0.5,0.8660254040) (0.5,0.5) (0.5,0.2886751347) 
\qbezier (1,0) (1,0.1339745962) (0.9330127020,0.25)
\qbezier (0.9330127020,0.25) (0.8660254039,0.3660254038) (0.75,0.4330127020)
\qbezier (0.75,0.4330127020) (0.6339745962,0.5) (0.5,0.5)
\qbezier (0.5,0.5) (0.3660254038,0.5) (0.25,0.4330127020)
\qbezier (0.25,0.4330127020) (0.1339745962,0.3660254039) (0.0669872980,0.25)
\qbezier (0.0669872980,0.25) (0,0.1339745962) (0,0)
\qbezier (1,0) (1,0.06698729810) (0.9665063510,0.125)
\qbezier (0.9665063510,0.125) (0.9330127020,0.1830127019) (0.875,0.2165063510)
\qbezier (0.875,0.2165063510) (0.8169872981,0.25) (0.75,0.25)
\qbezier (0.75,0.25) (0.6830127019,0.25) (0.625,0.2165063510)
\qbezier (0.625,0.2165063510) (0.5669872981,0.1830127020) (0.5334936490,0.125)
\qbezier (0.5334936490,0.125) (0.5,0.06698729810) (0.5,0)
\qbezier (0.5,0) (0.5,0.06698729810) (0.4665063510,0.125)
\qbezier (0.4665063510,0.125) (0.4330127020,0.1830127019) (0.375,0.2165063510)
\qbezier (0.375,0.2165063510) (0.3169872981,0.25) (0.25,0.25)
\qbezier (0.25,0.25) (0.1830127019,0.25) (0.125,0.2165063510)
\qbezier (0.125,0.2165063510) (0.0669872981,0.1830127020) (0.0334936490,0.125)
\qbezier (0.0334936490,0.125) (0,0.06698729810) (0,0)
\qbezier (0.5,0.2886751347) (0.44,0.27) (0.4,0.2)  
\qbezier (0.5,0.2886751347) (0.56,0.27) (0.6,0.2)  
\end{picture}}\end{equation}
We have observed that $\Gamma_1(4)\subset{\mathrm{SL}}(4,{\mathbb{Z}})$ has 
index~$12$. It follows that 
$$\{\pm{\mathrm{Id}}\}\times\Gamma_1(4)\subset{\mathrm{SL}}(2,{\mathbb{Z}})$$
has index~$6$ and, therefore, that $\Gamma_1(4)$ may be regarded as a subgroup
of ${\mathrm{PSL}}(2,{\mathbb{Z}})$ of index~$6$. Certainly,
$$\left\langle\mat{1}{1}{0}{1},\mat{1}{0}{4}{1}\right\rangle
\subseteq\Gamma_1(4).$$
Equality follows because, as subgroups of ${\mathrm{PSL}}(2,{\mathbb{Z}})$,
they have the same index of~$6$, as~(\ref{look_six_copies}) shows.
\end{proof}
\noindent It is usual to introduce another congruence subgroup of the modular
group
$$\Gamma_0(4)\equiv\left\{\mat{a}{b}{c}{d}\in{\mathrm{SL}}(2,{\mathbb{Z}})
\mid\mat{a}{b}{c}{d}\equiv\mat{*}{*}{0}{*}\bmod 4\right\}$$
but it has already occurred in our proof above as 
$\{\pm{\mathrm{Id}}\}\times\Gamma_1(4)$. 

In summary, the group ${\mathrm{SL}}(2,{\mathbb{R}})$ acts on the upper half
plane ${\mathcal{H}}$ by
$$\mat{a}{b}{c}{d}\tau\equiv\frac{a\tau+b}{c\tau+d}.$$
The resulting homomorphism 
${\mathrm{SL}}(2,{\mathbb{R}})\to{\mathrm{Biholo}}({\mathcal{H}})$ is a double 
cover, having $\{\pm{\mathrm{Id}}\}$ as kernel. The subgroup 
$\Gamma_0(4)\subset{\mathrm{SL}}(2,{\mathbb{R}})$ descends to
$$\Gamma_1(4)\subset{\mathrm{PSL}}(2,{\mathbb{R}})
={\mathrm{Biholo}}({\mathcal{H}}),$$
which acts discontinuously and without fixed points. The resulting mapping
$${\mathcal{H}}\longrightarrow
\raisebox{-3pt}{$\Gamma_1(4)$}\raisebox{-2pt}{\large$\backslash$}
{\mathcal{H}}=\frac{\mathcal{H}}
{\Big\{\tau\sim\tau+1,\displaystyle\tau\sim\frac{\tau}{4\tau+1}\Big\}}\cong
S^2\setminus\{0,1,\infty\}\equiv\Sigma$$
is an explicit (and conformal) realisation of the universal cover of the
thrice-punctured sphere~$\Sigma$.

Note that there is still a certain amount of mystery built into this
realisation, which can be traced back to our use of the non-constructive
Riemann mapping theorem at the start of Section~\ref{thrice_punctured}. This
mystery now shows up in our having two natural local co\"ordinates near the
South Pole. On the one hand, we may write $q=e^{2\pi i\tau}$, as we did for the
twice-punctured sphere, to obtain a local holomorphic co\"ordinate~$q$
replacing $\tau\sim\tau+1$ for $\tau=s+it$ as $t\uparrow\infty$. On the other
hand, we have, by construction, the global meromorphic co\"ordinate $z$ on the
sphere with the South Pole at~$z=0$. It follow that $z$ is a holomorphic
function of $q$ near $\{q=0\}$ and vice versa. For the moment, the relationship
between $z$ and $q$ is mysterious save that various key points coincide:
$$\begin{array}{c||c|c|c}
z&0&1&\infty\\ \hline
q&0&-1&1 
\end{array}.$$
It is clear, however, that $\Sigma$ acquires a {\em projective structure\/}: a
preferred set of local co\"ordinates related by M\"obius transformations. In
fact, it is better: we have $\tau$ defined up to
${\mathrm{PSL}}(2,{\mathbb{R}})$ freedom ({\em real\/} M\"obius
transformations).

\section{Puncture repair}
The main upshot of the reasoning in
Sections~\ref{thrice_punctured}--\ref{upper_half_plane} is a realisation of the
thrice-punctured Riemann sphere $\Sigma\equiv S^2\setminus\{0,1,\infty\}$ as
the upper half plane ${\mathcal{H}}$ modulo the action of $\Gamma_1(4)$, an
explicit subgroup of ${\mathrm{Aut}}({\mathcal{H}})$ acting properly
discontinuously and without fixed points. Furthermore, it is evident from this
construction, that $\Sigma$ {\em may\/} be compactified as the Riemann sphere
(using, for example, the co\"ordinate change $q=e^{2\pi i\tau}$). In fact, an
argument due to Ahlfors and Beurling~\cite{AB} shows that there are no other
conformal compactifications.
\begin{thm}\label{AB} Suppose $M$ is a compact Riemann surface 
with\/~$\Sigma\hookrightarrow M$ a conformal isomorphism onto an open subset 
of~$M$. Then $M$ must be conformal to the Riemann sphere 
with\/~$\Sigma\hookrightarrow S^2$ the standard embedding.
\end{thm}
\begin{proof} In fact, this is a local result as in the following picture,
$$\setlength{\unitlength}{.7cm}
\begin{picture}(6,3)(1.1,0.45)
\linethickness{0.4mm}
\qbezier(5.8,2.0)(5.8,2.3728)(4.9799,2.6364)
\qbezier(4.9799,2.6364)(4.1598,2.9)(3.0,2.9)
\qbezier(3.0,2.9)(1.8402,2.9)(1.0201,2.6364)
\qbezier(1.0201,2.6364)(0.2,2.3728)(0.2,2.0)
\qbezier(0.2,2.0)(0.2,1.6272)(1.0201,1.3636)
\qbezier(1.0201,1.3636)(1.8402,1.1)(3.0,1.1)
\qbezier(3.0,1.1)(4.1598,1.1)(4.9799,1.3636)
\qbezier(4.9799,1.3636)(5.8,1.6272)(5.8,2.0)
\put(3,2){\makebox(0,0){$\bullet$}}
\end{picture}\hspace{-12pt}\raisebox{26pt}{\large$\cong$}\qquad
\begin{picture}(7,4)(0,1.2)
\linethickness{0.1mm}
\qbezier(5.8,2.0)(5.8,2.3728)(4.9799,2.6364)
\qbezier(4.9799,2.6364)(4.1598,2.9)(3.0,2.9)
\qbezier(3.0,2.9)(1.8402,2.9)(1.0201,2.6364)
\qbezier(1.0201,2.6364)(0.2,2.3728)(0.2,2.0)
\linethickness{0.4mm}
\qbezier(0.2,2.0)(0.2,1.6272)(1.0201,1.3636)
\qbezier(1.0201,1.3636)(1.8402,1.1)(3.0,1.1)
\qbezier(3.0,1.1)(4.1598,1.1)(4.9799,1.3636)
\qbezier(4.9799,1.3636)(5.8,1.6272)(5.8,2.0)
\linethickness{0.25mm}
\qbezier(5.8,2)(5.8,3)(4,4)
\qbezier(4,4)(2.2,5)(1,5)
\qbezier(1,5)(.2,5)(.2,2)
\qbezier(1.2,4.4)(1.7,4.5)(2.2,4.2)
\qbezier(.9,4.6)(1.7,3.8)(2.6,4.3)
\qbezier(.3,3.5)(1,3.3)(2.8,2.8)
\qbezier(2.8,2.8)(4.6,2.3)(5,3.3)
\linethickness{0.1mm}
\qbezier(.3,3.5)(2.5,4.3)(5,3.3)
\put(-.3,2.75){\makebox(0,0){U$\left\{\rule{0pt}{18pt}\right.$}}
\put(7.1,3.45){\makebox(0,0){$\left.\rule{0pt}{34pt}\right\}V\!\subset\!M$}}
\end{picture}$$
taken from~\cite{eg}. The punctured open disc is assumed to be conformally
isomorphic to the open set $U$ (but nothing is supposed concerning the boundary
$\partial{U}$ of $U$ in~$V$). We conclude that $V$ is conformally the disc and
$U\hookrightarrow V$ the punctured disc, tautologically included. To see this,
we calculate in polar co\"ordinates $(r,\theta)$ on the unit disc. We know
that there is a smooth positive function $\Omega(r,\theta)$ defined for $r>0$
so that the metric $\Omega(r,\theta)^2(dr^2+r^2d\theta^2)$ smoothly 
extends from $U$ to~$V$. We will encounter a contradiction if 
$\partial U$ contains two or more points since, in this case, the concentric 
curves $\{r=\epsilon\}$, as $\epsilon\downarrow 0$, have length bounded away 
from zero in the metric $\Omega(r,\theta)^2(dr^2+r^2d\theta^2)$. More 
explicitly, 
$$\int_0^{2\pi}\Omega(r,\theta)r\,d\theta$$
is bounded away from zero as $r\downarrow 0$. On the other hand, the area of
the region $\{0<r<\epsilon\}$ in $V$ is estimated by Cauchy-Schwarz as
$$\int_0^\epsilon\!\!\int_0^{2\pi}\Omega^2d\theta\,r\,dr\geq
\frac1{2\pi}\int_0^\epsilon\!\left[\int_0^{2\pi}\Omega\,d\theta\right]^2\!r\,dr
=\frac1{2\pi}\int_0^\epsilon\!
\left[\int_0^{2\pi}\Omega r\,d\theta\right]^2\frac{dr}r$$
and is therefore forced to be infinite.\end{proof} 

Otherwise said, there is no difference between the Riemann sphere, either 
{\em marked\/} at $\{0,1,\infty\}$ or {\em punctured\/} there. Thus, it makes
\mbox{intrinsic} sense on $\Sigma\equiv S^2\setminus\{0,1,\infty\}$ to 
consider holomorphic $1$-forms that are \mbox{restricted} from meromorphic
$1$-forms on $S^2$ with poles only at the marks. Of special interest is the
space (in traditional arcane notation)
$${\mathcal{M}}_2(\Gamma_0(4))\equiv\left\{\!\!\begin{tabular}{l}
holomorphic $1$-forms $\omega$ on $\Sigma$ extending\\ 
meromorphically to $S^2$ with, at worst,\\
only simple poles at $0,1,\infty$.
\end{tabular}\!\!\right\}.$$
\begin{thm} There is a canonical isomorphism
$${\mathcal{M}}_2(\Gamma_0(4))\cong\{(a,b,c)\in{\mathbb{C}}^3\mid a+b+c=0\}.$$
\end{thm}
\begin{proof} The isomorphism is given by
$$\omega\longmapsto({\mathrm{Res}}_{z=0\,}\omega,
{\mathrm{Res}}_{z=1\,}\omega,{\mathrm{Res}}_{z=\infty\,}\omega),$$
with $a+b+c=0$ being a consequence of the Residue Theorem.
\end{proof}
In particular, there is the special meromorphic $1$-form
$$\frac{dz}{z},\enskip\mbox{holomorphic save for}
\enskip\Big\{\!\begin{array}{l}
\mbox{simple poles only at $0$ and $\infty$},\\
\mbox{residue}=1\mbox{ at }0.\end{array}$$

%
%
%
%
\section{Automorphisms of the thrice-punctured sphere} 
By the Ahlfors-Beurling Theorem, automorphisms of $\Sigma$ correspond to
permutations of $\{0,1,\infty\}$ and there are two particular ones that we 
shall find useful. Firstly, since
$$\mat{0}{-1/2}{2}{0}\mat{a}{b}{c}{d}
=\mat{d}{-c/4}{-4b}{a}\mat{0}{-1/2}{2}{0},$$
it follows that
\begin{equation}\label{auto1}\tau\mapsto-1/{4\tau}\end{equation}
induces an automorphism of~$\Sigma$. In the $z$-co\"ordinate, it is the one
that swops $0$ and $\infty$ but fixes~$1$, namely $z\mapsto 1/z$. 

Secondly, since
$$\mat{1}{1/2}{0}{1}\mat{a}{b}{c}{d}
=\mat{a+c/2}{b+(d-a)/2-c/4}{c}{d-c/2}\mat{1}{1/2}{0}{1},$$
it follows that
\begin{equation}\label{auto2}\tau\mapsto\tau+1/2\end{equation}
is the automorphism of $\Sigma$ that swops $z=1$ and $z=\infty$ whilst
fixing~$0$. Close to $q=0$, we recognise it as $q\mapsto -q$. In the
$z$-co\"ordinate, it is
$$z\mapsto z/(z-1).$$

\section{The normal distribution}
At this point, rather bizarrely, it is useful to discuss the normal 
distribution
$$f(x)\equiv e^{-\pi x^2}$$
and its well-known invariance under the Fourier transform
$$\widehat{f}(\xi)\equiv\int_{-\infty}^\infty f(x)e^{-2\pi\xi x}dx 
= e^{-\pi\xi^2}.$$
More generally, integration by substitution shows that 
\begin{equation}\label{FT}f(x)=e^{-2\pi tx^2}\Longrightarrow 
\widehat{f}(\xi)=\frac1{\sqrt{2t}}e^{-(\pi/2t)\xi^2}\end{equation}
for any $t>0$. The Poisson summation formula says that 
$$\sum_{n\in{\mathbb{Z}}}f(n)=\sum_{n\in{\mathbb{Z}}}\widehat{f}(n)$$
for $f:{\mathbb{R}}\to{\mathbb{R}}$ a suitably well-behaved function (for 
example, one that lies in Schwartz space). For $f(x)=e^{-\pi tx^2}$, as in 
(\ref{FT}), we find that
\begin{equation}\label{magic}\sum_{n\in{\mathbb{Z}}}e^{-2\pi t n^2}
=\frac1{\sqrt{2t}}\sum_{n\in{\mathbb{Z}}}e^{-(\pi/2t)n^2}.\end{equation}

\section{A miracle}
An outrageous suggestion is to view the formal power series~(\ref{theta}) as
defining a holomorphic function of the complex variable~$q$ (now called {\em
Jacobi's theta function}). Clearly, it is convergent for $\{|q|<1\}$. Hence,
setting $q=e^{2\pi i \tau}$, we obtain a holomorphic function of $\tau$ for
$\tau\in{\mathcal{H}}$. Then a miracle occurs: 
\begin{thm}\label{miracle} For
$\tau\in{\mathcal{H}}$, we have
$$(\theta(-1/4\tau))^4=-4\tau^2(\theta(\tau))^4.$$
Equivalently, if we define $\phi:{\mathcal{H}}\to{\mathcal{H}}$ by 
$$\phi(\tau)\equiv -1/4\tau$$
and consider the holomorphic $1$-form
$\Theta\equiv(\theta(\tau))^4d\tau,$
then
\begin{equation}\label{theta_transformation}
\phi^*\Theta=-\Theta.\end{equation}
\end{thm}
\begin{proof} When $\tau$ lies on the imaginary axis, i.e.~$\tau=it$ for $t>0$,
$$\theta(\tau)=\sum_{n\in{\mathbb{Z}}}q^{n^2}
=\sum_{n\in{\mathbb{Z}}}e^{-2\pi tn^2}$$
whilst 
$$\theta(-1/4\tau)=\sum_{n\in{\mathbb{Z}}}e^{-2\pi(1/4t)n^2}
=\sum_{n\in{\mathbb{Z}}}e^{-(\pi/2t)n^2}$$
so (\ref{magic}) says that 
$$\theta(-1/4\tau)=\sqrt{-2i\tau}\theta(\tau),\quad\mbox{whence}\quad
(\theta(-1/4\tau))^4=-4\tau^2(\theta(\tau))^4,$$
along the imaginary axis. The transformation (\ref{theta_transformation}) now
holds for all $\tau\in{\mathcal{H}}$ by analytic continuation.
\end{proof}
Notice that the transformation $\phi$
has already made its appearance (\ref{auto1}) as inducing an automorphism
of~$\Sigma$, the thrice-punctured sphere. If we also introduce
$T:{\mathcal{H}}\to{\mathcal{H}}$ by
$$T(\tau)\equiv\tau+1,$$
then it is clear that
$T^*\theta=\theta$ and $T^*d\tau=d\tau$. Hence, we see that
\begin{equation}\label{T-invariance}T^*\Theta=\Theta.\end{equation}
Finally, to obtain a geometric interpretation of (\ref{theta_transformation})
we note that
$$R\equiv\phi\circ T^{-1}\circ\phi$$
is given by 
$$R(\tau)=\frac{\tau}{4\tau+1}$$
and recall that $R$ and $T$ together generate $\Gamma_0(4)$. Note that
$R^*\Theta=\Theta$ in accordance with (\ref{theta_transformation})
and~(\ref{T-invariance}). Putting all this together, we have proved the
following. 
\begin{thm} The holomorphic $1$-form
$\Theta\equiv(\theta(\tau))^4d\tau$ descends to the thrice-punctured
sphere~$\Sigma$ and, under the automorphism $\phi:\Sigma\to\Sigma$, satisfies
$\phi^*\Theta=-\Theta$.
\end{thm}
\begin{cor} In the usual $z$-co\"ordinate on the thrice-punctured sphere,
$$\Theta=\frac{dz}{2\pi iz}.$$
\end{cor}
\begin{proof} {From} $q=e^{2\pi i\tau}$ we see that $dq=2\pi i qd\tau$ and so
$$\Theta=\frac1{2\pi iq}\left(1+8q+24q^2+32q^3+\cdots\right)dq$$
near $q=0$ and, in particular, meromorphically extends through~$q=0$, having a
simple pole there with residue $1/2\pi i$. This is a co\"ordinate-free 
statement and so 
also applies in the $z$-co\"ordinate:
$$\Theta=\frac1{2\pi iz}\left(1+\cdots\right)dz.$$
Recall that in the $z$-co\"ordinate, the automorphism $\phi$ interchanges $z=0$
with $z=\infty$ whilst fixing $z=1$. The relation $\phi^*\Theta=-\Theta$,
implies that $\Theta$ also has a pole at $z=\infty$ with residue $-1/2\pi i$.
Finally, the behaviour of $\Theta$ at $z=1$ may be investigated by means of the
automorphism (\ref{auto2}), let us call it~$\psi$, which swops $z=1$ and
$z=\infty$ whilst fixing $z=0$. In particular, we may easily compare $\Theta/i$
along the imaginary $\tau$-axis $\{\tau=it\}$ with its behaviour along the
translated axis $\{\tau=1/2+it\}$:
$$\begin{array}{rcl}\Theta/i&=&\left(1+8q+24q^2+32q^3+\cdots\right)dt\\[5pt]
\psi^*\Theta/i&=&\left(1-8q+24q^2-32q^3+\cdots\right)dt\end{array}
\quad\mbox{where }q=e^{-2\pi t}.$$
It is clear that $\Theta(it)$ has only a simple pole at $t=0$. But
$\Theta(\tau)$ is real-valued when $\mathrm{Re}(\tau)=0$ or
$\mathrm{Re}(\tau)=1/2$, and the $q$-expansion coefficients are all
non-negative, so $\Theta(1/2+it)$ is dominated by $\Theta(it)$ as $t\to 0^{+}$.
The possibility of an essential singularity is excluded by the observation that
the intersection of any semicircle centred at $\tau=1/2$ with an appropriately
chosen fundamental domain containing $\{1/2+it \mid t\geq 0\}$ is a finite
curve, so the maximal value of $\Theta(\tau)$, as $\tau$ runs along the
semicircle, is bounded by $\Theta(is)$ for some real $s$. So the behaviour of
$\Theta$ at $z=1$ is certainly no worse than the behaviour at $z=0$.

In summary, the holomorphic $1$-form $\Theta$ on $\Sigma\equiv
S^2\setminus\{0,1,\infty\}$ enjoys a meromorphic extension to $S^2$ with
\begin{itemize}
\item a simple pole at $z=0$ with residue $1/2\pi i$,
\item a simple pole at $z=\infty$ with residue $-1/2\pi i$,
\item at worse at simple pole at $z=1$.
\end{itemize}
By the residue theorem, the sum of the residues of any meromorphic 
$1$-form on any Riemann surface is zero. It follows that $\Theta$ has poles 
only at $z=0$ and $z=\infty$. Having identified precisely two poles, it cannot 
have any zeros. At this point $\Theta$ is determined as stated.   
\end{proof}

\section{An Eisenstein series}
Introduce
$$G_4(\tau)\equiv
\sum_{(c,d)\in{\mathbb{Z}}^2\setminus\{(0,0)\}}\frac1{(c\tau+d)^4}$$
and, by absolute convergence, observe that
\begin{equation}\label{G4transformation}
G_4\Big(\frac{a\tau+b}{c\tau+d}\Big)=(c\tau+d)^4G_4(\tau),\quad
\mbox{for}\enskip\mat{a}{b}{c}{d}\in{\mathrm{SL}}(2,{\mathbb{Z}}).
\end{equation}
\begin{thm}\label{G4series}
\begin{equation}\label{expandG4}
 G_4(q)=\frac{\pi^4}{45}
\Big(1+240\sum_{n=1}^\infty\sigma_3(n)q^n\Big),\end{equation}
where $\sigma_3(n)\equiv\sum_{d|n}d^3$ (and recall that $q=e^{2\pi i\tau}$). 
\end{thm}
\begin{proof} This is a straightforward application of (\ref{another_rite}):
\begin{align*}
G_{4}(\tau)
&=\!\sum_{\substack{d=-\infty\\d\neq 0}}^{\infty}\frac{1}{d^4}
 +\!\!\sum_{\substack{c=-\infty\\c\neq 0}}^{\infty}
  \sum_{d=-\infty}^{\infty}\frac{1}{(c\tau+d)^4}
=2\zeta(4)
 +2\sum_{c=1}^{\infty}\sum_{d=-\infty}^{\infty}\frac{1}{(c\tau+d)^4}\\
&=\frac{\pi^4}{45}
 +2\sum_{c=1}^{\infty}
   \left(\frac{8\pi^4}{3}\sum_{m=1}^{\infty}m^3 e^{2\pi i cm\tau}\right)
\quad\mbox{(from (\ref{another_rite}))}\\
&=\frac{\pi^4}{45}
\left(1+240\sum_{m=1}^{\infty}\sigma_{3}(m)e^{2\pi i m\tau}\right).\qedhere
\end{align*}
\end{proof}
Following Ramanujan, let
\begin{equation}\label{M}M(q)\equiv 1+240\sum_{n=1}^\infty\sigma_3(n)q^n
\end{equation}
and, as a consequence of (\ref{G4transformation}) and (\ref{expandG4}), observe
that
\begin{equation}\label{Mtransformation}
M(\tau+1)=M(\tau)\quad\mbox{and}\quad M(-1/\tau)=\tau^4M(\tau).\end{equation}

\section{The Ramanujan ODE} 
Following Ramanujan, let
\begin{equation}\label{L}
L(q)\equiv 1-24\sum_{n=1}^\infty\sigma(n)q^n
\end{equation}
defined for $\{|q|<1\}$. The following identity was proved by Ramanujan
\cite[identities (17), (27), (28), and (30)]{R}, as a corollary of his
straight\-forward but inspired proof of a certain identity between Lambert
series. These Lambert series identities were elucidated by van~der~Pol
\cite{V}, who showed that they ultimately derive from the product formula and
transformation formula for Jacobi's theta function. A direct combinatorial
proof is due to Skoruppa~\cite{S}. \begin{thm} As (formal) power series,
\begin{equation}\label{RSode}12q\frac{dL}{dq}-L^2+M=0.\end{equation}
\end{thm}
As usual, by setting $q=e^{2\pi i\tau}$, we may view
$L$ as a holomorphic function $L(\tau)$ for $\tau\in{\mathcal{H}}$. A change of 
variables gives
\begin{equation}\label{RSbis}\frac6{\pi i}\frac{dL}{d\tau}-L^2+M=0,
\end{equation}
an equivalent statement to (\ref{RSode}). Locally, we may write
\begin{equation}\label{potential}
L(\tau)=-\frac6{\pi i}\frac{g^\prime(\tau)}{g(\tau)}\end{equation}
and (\ref{RSbis}) becomes $g^{\prime\prime}+\frac{\pi^2}{36}Mg=0$.
Thus, we are led to consider
\begin{equation}\label{RSlinear}
y^{\prime\prime}+\frac{\pi^2}{36}My=0\end{equation} for
$y:{\mathcal{H}}\to{\mathbb{C}}$ a holomorphic function and (\ref{RSode}) says
that $y(\tau)=g(\tau)$ is a solution of (\ref{RSlinear}). We may investigate
the solutions of the {\em linear\/} equation (\ref{RSlinear}) quite explicitly.
Firstly, we may figure out much more about $g(\tau)$ as follows.

\begin{lemma}\label{g} We may take
$$\begin{array}{rcl}g(\tau)
&=&\displaystyle 
e^{-\pi i\tau/6}\exp\Big(2\sum_{n=1}^\infty\frac{\sigma(n)}{n}q^n\Big)\\[14pt]
&=&e^{-\pi i\tau/6}\big(1+2q+5q^2+10q^3+20q^4+36q^5+65q^6+\cdots\big),
\end{array}$$
a globally defined holomorphic function 
${\mathcal{H}}\to{\mathbb{C}}\setminus\{0\}$. 
\end{lemma}
\begin{proof}
Of course, the function $g(\tau)$ is locally defined by (\ref{potential}) up to
a constant. As a global Ansatz, let us try
$$g(\tau)=e^{-\pi i\tau/6}\psi(q),\quad\mbox{for}\enskip q=e^{2\pi i\tau}$$
and $\psi:\{|q|<1\}\to{\mathbb{C}}\setminus\{0\}$ holomorphic. Substituting
this form of $g$ into (\ref{potential}) gives
\begin{equation}\label{psiODE}
\psi-12q\frac{d\psi}{dq}=L\psi=\psi-24\psi\sum_{n=1}^\infty\sigma(n)q^n
\end{equation}
so 
$$\frac{d}{dq}\log\psi=\frac1\psi\frac{d\psi}{dq}
=2\sum_{n=1}^\infty\sigma(n)q^{n-1}
=2\frac{d}{dq}\sum_{n=1}^\infty\frac{\sigma(n)}nq^n$$
and, normalising $\psi(q)$ by $\psi(0)=1$, conclude that
$$\log\psi=2\sum_{n=1}^\infty\frac{\sigma(n)}nq^n.$$
Evidently, this power series converges for $|q|<1$ and we are done.
\end{proof}
As an aside, we note that the resulting power series expansion
$$\psi(q)=\sum_{n=0}^\infty b_nq^n=1+2q+5q^2+10q^3+20q^4+36q^5+65q^6+\cdots,$$
where, as one obtains easily from~(\ref{psiODE}), 
\begin{equation}\label{b-recursion}
b_0=1,\quad b_n=\frac2n\sum_{k=1}^n\sigma(k)b_{n-k},\enskip\mbox{for} \enskip
n\geq1,\end{equation}
has integer coefficients. Indeed, the generating function of $\sigma$ is the
$q$-expansion of a Lambert series 
$$\sum_{n=1}^{\infty}\sigma(n)q^n=\sum_{n=1}^{\infty}\frac{nq^n}{1-q^n},$$
which, upon rewriting, assumes the form
$$\sum_{n=1}^{\infty}\frac{nq^n}{1-q^n}
=q\frac{d}{dq}\sum_{n=1}^{\infty}\log\left(\frac{1}{1-q^n}\right)
=q\frac{d}{dq}\log\prod_{k=1}^{\infty}\frac{1}{1-q^k}.$$
But the $q$-expansion of this infinite product is well-known. It is the
generating function of the manifestly integral partition numbers~$p(k)$:
$$\prod_{k=1}^{\infty}\frac{1}{1-q^k}=\sum_{k=0}^{\infty}p(k)q^k\equiv P(q).$$
Returning to (\ref{psiODE}), we find that $\psi$ satisfies 
$$\frac{d}{dq}\left(\log\psi(q)-2\log P(q)\right)=0,$$
and, recalling that $P(0)=1$, we find that~$\psi=P^2$.

Let ${\mathbb{S}}$ denote the solution space of~(\ref{RSlinear}). As
${\mathcal{H}}$ is simply-connected, we conclude that ${\mathbb{S}}$ is
two-dimensional and in Lemma~\ref{g} we have already found one non-zero element 
in~${\mathbb{S}}$. To complete our understanding of ${\mathbb{S}}$ it suffices 
to find another linearly independent element:
\begin{lemma}\label{h}
There is a convergent power series
$$\textstyle\phi(q)
=1+\frac{10}{7}q+\frac{365}{91}q^2+\frac{13610}{1729}q^3+\frac{135701}{8645}q^4
+\frac{7419742}{267995}q^5+\cdots\quad\mbox{for}\enskip|q|<1$$
so that $h(\tau)\equiv e^{\pi i \tau/6}\phi(q)$ is in\/~${\mathbb{S}}$.
\end{lemma}
\begin{proof} We try $y(\tau)=e^{\pi i \tau/6}\phi(q)$ as an Ansatz 
in~(\ref{RSlinear}). A calculation shows that (\ref{RSlinear}) reduces to
$$6q^2\frac{d^2\phi}{dq^2}+7q\frac{d\phi}{dq}
=10\sum_{n=1}^\infty\sigma_3(n)q^n,$$
whereas substituting $h(\tau)=e^{-\pi i\tau/6}\psi(q)$ instead, gives
$$6q^2\frac{d^2\psi}{dq^2}+5q\frac{d\psi}{dq}
=10\sum_{n=1}^\infty\sigma_3(n)q^n.$$
Each of these gives a recursion relation for the coefficients of a formal 
power series for the function in question, namely 
$$\phi(q)=\sum_{n=0}^\infty a_nq^n\qquad\psi(q)=\sum_{n=0}^\infty b_nq^n$$
where $a_0=b_0=1$ and, for $n\geq 1$,
$$a_n=\frac{10}{n(6n+1)}\sum_{k=1}^n\sigma_3(k)a_{n-k}\qquad
b_n=\frac{10}{n(6n-1)}\sum_{k=1}^n\sigma_3(k)b_{n-k}.$$
By Lemma~\ref{g}, we know that the power series $\sum_{n=0}^\infty b_nq^n$
converges for $|q|<1$ (and, from this formal point of view, the content of
(\ref{RSode}) is that the recursion relation (\ref{b-recursion}) yields the
same coefficients~$b_n$). {From} these recurrence relations it is clear, by
induction, that $0<a_n\leq b_n$. It follows that $\sum_{n=0}^\infty a_nq^n$
also converges for $|q|<1$ and we are done.\end{proof} 
In summary,
Lemmata~\ref{g} and~\ref{h} give us a basis for ${\mathbb{S}}$ of the form
$$\begin{array}{rcr}
g(\tau)&=&e^{-\pi i\tau/6}\psi(q)\phantom{,}\\[4pt]
h(\tau)&=&e^{\pi i\tau/6}\phi(q),
\end{array}\quad\mbox{where}\enskip q=e^{2\pi i\tau}$$
and $\phi(q),\psi(q)$ are holomorphic functions on the unit disc $\{|q|<1\}$.
Also notice that both $\psi(e^{2\pi i\tau})$ and $\phi(e^{2\pi i\tau})$ are
strictly positive along the imaginary axis~$\{\tau=it|t>0\}$
in~${\mathcal{H}}$. In particular, we conclude that $h(i)\not=0$.

\begin{thm}\label{firstBGG} 
The equation \eqref{RSlinear} is projectively invariant.
\end{thm}
\begin{proof} Firstly, we must explain what the phrase `projectively invariant'
means. There is no local structure in the conformal geometry 
of~${\mathcal{H}}$ (an $n$-dimensional complex manifold is locally 
biholomorphic to ${\mathbb{C}}^n$; end of story). Globally, however, the group 
${\mathrm{SL}}(2,{\mathbb{R}})$ acts conformally on~${\mathcal{H}}$ and this 
may be recorded as local information on~${\mathcal{H}}$, specifically as a 
collection of preferred local co\"ordinates, namely $\tau$ and its translates
$$\frac{a\tau+b}{c\tau+d}\quad\mbox{for}
\enskip\mat{a}{b}{c}{d}\in{\mathrm{SL}}(2,{\mathbb{R}}).$$
Roughly speaking, this is a `projective structure.' In any case, to say that 
(\ref{RSlinear}) is `projectively invariant' is to say that it respects the 
action of ${\mathrm{SL}}(2,{\mathbb{Z}})$. For this to be true we decree 
that 
\begin{equation}\label{weight1}(A^{-1}g)(\tau)\equiv(c\tau+d)g(A\tau),
\quad\mbox{for}\enskip A=\mat{a}{b}{c}{d}\in{\mathrm{SL}}(2,{\mathbb{R}}).
\end{equation}
(In the language of projective differential geometry $g$ is a `projective 
density of weight~$1$.') {From} (\ref{G4transformation}), 
(\ref{expandG4}), and (\ref{M}), we already know that 
$$M\Big(\frac{a\tau+b}{c\tau+d}\Big)=(c\tau+d)^4M(\tau),
\quad\mbox{for}\enskip\mat{a}{b}{c}{d}\in{\mathrm{SL}}(2,{\mathbb{Z}})$$
and so it suffices to show that
$$\frac{d^2}{d\tau^2}\left[(c\tau+d)g\Big(\frac{a\tau+b}{c\tau+d}\Big)\right]
=\frac1{(c\tau+d)^3}\frac{d^2g}{d\tau^2}\Big(\frac{a\tau+b}{c\tau+d}\Big),$$
which is an elementary consequence of the chain rule.
\end{proof}
Recall that ${\mathbb{S}}$, the solution space of~(\ref{RSlinear}), is
two-dimensional. In accordance with Theorem~\ref{firstBGG}, the group
${\mathrm{SL}}(2,{\mathbb{Z}})$, generated by
\begin{equation}\label{TandS}T\equiv\mat{1}{1}{0}{1}\quad\mbox{and}\quad 
S\equiv\mat{0}{-1}{1}{0},\end{equation}
is represented on~${\mathbb{S}}$. More specifically, if $g(\tau)$
solves~(\ref{RSlinear}) then, according to~(\ref{weight1}), so do
$$(Tg)(\tau)\equiv g(\tau-1)\quad\mbox{and}\quad
(Sg)(\tau)\equiv -\tau g(-1/\tau).$$

\begin{thm}\label{transformation_of_L}
The holomorphic function $L:{\mathcal{H}}\to{\mathbb{C}}$ satisfies
\begin{equation}\label{quasimodular}
L\Big(\frac{a\tau+b}{c\tau+d}\Big)=(c\tau+d)^2L(\tau)+\frac6{\pi i}c(c\tau+d)
\end{equation}
for $\mat{a}{b}{c}{d}\in{\mathrm{SL}}(2,{\mathbb{Z}})$. 
\end{thm}
\begin{proof} 
It suffices to prove (\ref{quasimodular}) for the generators $T$ and $S$ of
${\mathrm{SL}}(2,{\mathbb{Z}})$, specifically that
$$L(\tau+1)=L(\tau)\quad\mbox{and}
\quad L(-1/\tau)=\tau^2L(\tau)+6\tau/\pi i.$$
The first of these holds by Lemma~\ref{g}, which implies that 
$Tg=e^{\pi i/6}g$. To establish the second identity, it suffices to show that
$Sg=\beta g$ for some constant $\beta$: if $-\tau g(-1/\tau)=\beta g(\tau)$,
then
$$\beta g(-1/\tau)=g(\tau)/\tau\enskip\Rightarrow\enskip
\beta g^\prime(-1/\tau)=\tau g^\prime(\tau)-g(\tau)$$
so
$$\frac{\beta g^\prime(-1/\tau)}{g(\tau)}
=\frac{\tau g^\prime(\tau)}{g(\tau)}-1.$$
Therefore
$$\frac{g^\prime(-1/\tau)}{\tau g(-1/\tau)}
=\frac{\tau g^\prime(\tau)}{g(\tau)}-1$$
and so
$$-\frac{6}{\pi i}\frac{g^\prime(-1/\tau)}{\tau g(-1/\tau)}
=-\frac{6}{\pi i}\frac{\tau g^\prime(\tau)}{g(\tau)}+\frac{6}{\pi i};$$
in other words, from (\ref{potential}),
$$\frac{L(-1/\tau)}\tau=\tau L(\tau)+\frac6{\pi i},$$
as required. To finish the proof, let us consider the action of
${\mathrm{SL}}(2,{\mathbb{Z}})$ on~${\mathbb{S}}$. If $Sg\not=\beta g$, then 
we may set $f\equiv Sg$ to obtain $\{f,g\}$ as a basis of~${\mathbb{S}}$. By 
construction
$$S\vect{f}{g}=\mat{0}{-1}{1}{0}\vect{f}{g}.$$
By Lemma~\ref{g}, we already know that $Tg=e^{\pi i/6}g$ and, from 
Lemma~\ref{h}, we know that the action of $T$ on ${\mathbb{S}}$ is 
diagonalisable with the other eigenvalue being~$e^{-\pi i/6}$. In other words
$$T\vect{f}{g}=\mat{e^{-\pi i/6}}{\alpha}{0}{e^{\pi i/6}}\vect{f}{g}$$
for some constant~$\alpha$. In
${\mathrm{SL}}(2,{\mathbb{Z}})$, the matrices (\ref{TandS}) satisfy the 
relations
$$S^2=-{\mathrm{Id}}\quad\mbox{and}\quad(ST)^3=-{\mathrm{Id}}.$$
These same relations must hold for their action on ${\mathbb{S}}$. For $S$
this is evident and for $T$ we conclude that $\alpha=1$. Therefore, since
$$1+ie^{\pi i/6}=ie^{-\pi i/6}$$
we find that
$$T(f+ig)=Tf+iTg=e^{-\pi i/6}f+g+ie^{\pi i/6}g
=e^{-\pi i/6}(f+ig).$$
However, in Lemma~\ref{h}, we already found in $h$ an eigenvector for the
action of $T$ on~${\mathbb{S}}$ with eigenvalue~$e^{-\pi i/6}$. It follows that
\begin{equation}\label{multiple_of_h}f(\tau)+ig(\tau)=Ch(\tau)\end{equation}
for some constant~$C$. We have already observed that $h(i)\not=0$ whereas,
substituting $\tau=i$ into $f=Sg$, we find that
$$\big[f(\tau)=-\tau g(-1/\tau)\big]|_{\tau=i}\enskip\Rightarrow\enskip
f(i)=-ig(i)\enskip\Rightarrow\enskip\big[f+ig]|_{\tau=i}=0.$$
Therefore, the only option in (\ref{multiple_of_h}) is that~$C=0$ and 
so $f+ig\equiv 0$.
Hence, assuming that $Sg\not=\beta g$ we have found that $Sg=-ig$. This 
contradiction finishes the proof.
\end{proof}
\begin{cor}\label{Xi} The holomorphic $1$-form
$$\big(L(\tau)-L(\tau+1/2)\big)d\tau$$
is $\Gamma_0(4)$-invariant.
\end{cor}
\begin{proof} We need only check invariance under the generators of 
$\Gamma_0(4)$:
$$\tau\mapsto\tau+1\quad\mbox{and}\quad\tau\mapsto\frac\tau{4\tau+1}.$$
The first of these is clear since $L(\tau+1)=L(\tau)$. For the second, we may 
use Theorem~\ref{transformation_of_L} immediately to conclude that
$$L\Big(\frac{\tau}{4\tau+1}\Big)
=(4\tau+1)^2L(\tau)+\frac{24}{\pi i}(4\tau+1)$$
but also that
$$\begin{array}{rcl}\displaystyle L\Big(\frac\tau{4\tau+1}+\frac12\Big)
&\!\!\!\!=\!\!\!\!&\displaystyle 
L\Big(\frac{3(\tau+1/2)-1}{4(\tau+1/2)-1}\Big)\\[10pt]
&\!\!\!\!=\!\!\!\!&\displaystyle\big(4(\tau+1/2)-1\big)^2L(\tau+1/2)
+\frac{24}{\pi i}(4(\tau+1/2)-1)\\[8pt]
&\!\!\!\!=\!\!\!\!&\displaystyle
(4\tau+1)^2L(\tau+1/2)+\frac{24}{\pi i}(4\tau+1).
\end{array}$$
Subtracting these identities gives
$$L\Big(\frac{\tau}{4\tau+1}\Big)-L\Big(\frac\tau{4\tau+1}+\frac12\Big)
=(4\tau+1)^2\Big(L\big(\tau\big)-L\big(\tau+\frac12\big)\Big).$$
But
$$d\Big(\frac\tau{4\tau+1}\Big)
=\frac{(4\tau+1)d\tau-4\tau d\tau}{(4\tau+1)^2}
=\frac1{(4\tau+1)^2}d\tau,$$
the factor of $(4\tau+1)^2$ cancels, and we are done.
\end{proof}
\begin{lemma}\label{cusp} Suppose $\xi(\tau)$ is a holomorphic function
${\mathcal{H}}\to{\mathbb{C}}$ and let $q=e^{2\pi i\tau}$. In order that
$\xi(\tau)d\tau$ extend to a meromorphic differential form on the unit disc
$\{|q|<1\}$ with at worse a simple pole at $q=0$, it
is necessary and sufficient that
\begin{itemize}
\item $\xi(\tau+1)=\xi(\tau),\;\forall\tau\in{\mathcal{H}}$,
\item $\xi(\tau)$
is bounded on the rectangle $\{\tau=x+iy\,|\,0\leq x\leq 1,y\geq 1\}$.
\end{itemize}
\end{lemma}
\begin{proof}The first condition ensures that $\xi(\tau)$ is, in fact, a
holomorphic function of~$q$ and then, since 
$q=e^{2\pi i\tau}=e^{-2\pi y}e^{2\pi ix}$ the second condition says that
$\xi(q)$ is bounded on the disc $\{|q|<e^{-2\pi}\}$ at which point 
Riemann's removable singularities theorem implies that $\xi(q)$ extends
holomorphically across the origin: $\xi(q)=a+bq+\cdots$. Therefore,
$$q=e^{2\pi i\tau}\enskip\Rightarrow\enskip dq=2\pi i qd\tau
\enskip\Rightarrow\enskip\xi(\tau)d\tau=\frac1{2\pi i}\Big(\frac{a}{q}+
b+\cdots\Big)dq,$$
as required.
\end{proof}
Now consider the holomorphic $1$-form 
$$\Xi\equiv\big(L(\tau)-L(\tau+1/2)\big)d\tau\enskip\mbox{on}\enskip
{\mathcal{H}}.$$ With 
$q=e^{2\pi i\tau}$, as usual, it follows from the definition~(\ref{L}) of $L$ 
that 
$$L(\tau)-L(\tau+1/2)
=-48\big(q+4q^3+6q^5+\cdots\big)$$
and so $\Xi=-\frac{24}{\pi i}(1+4q^2+6q^4+\cdots)dq$ and, in particular,
extends holomorphically across $q=0$. Now we ask what happens at the cusps, a 
sensible question in view of Corollary~\ref{Xi}. 

The change of co\"ordinates 
$\tau=-1/4\widetilde\tau$ sends our usual fundamental domain for 
$\Gamma_0(4)$ into itself whilst sending
$$0\mapsto \infty,\quad
1/2\mapsto -1/2,\quad
\infty\mapsto 0,\quad
-1/2\mapsto 1/2$$
(it's a half turn about $i/2$ in the hyperbolic metric on~${\mathcal{H}}$).
In order to figure out the behaviour of $\Xi$ let us firstly consider the 
holomorphic $1$-form $\xi\equiv L(\tau)d\tau$. We may view it in the 
co\"ordinate $\widetilde\tau$:
$$\xi=L(-1/4\tilde\tau)d(-1/4\widetilde\tau)
=\frac{L(-1/4\widetilde\tau)}{4\widetilde\tau^2}d\widetilde\tau$$
and employ Theorem~\ref{transformation_of_L} to conclude that
$$\xi
=\frac{16\widetilde\tau^2L(4\widetilde\tau)+24\widetilde\tau/\pi i}
{4\widetilde\tau^2}d\widetilde\tau
=\Big(4L(4\widetilde\tau)+\frac{6}{\pi i\widetilde\tau}\Big)d\widetilde\tau.$$
Of course, whilst $4L(4\widetilde{\tau})$ is periodic under $\widetilde{\tau}
\mapsto \widetilde{\tau}+1$, $6/\pi i \widetilde{\tau}$ is not. Thus, the first
stipulation of Lemma~\ref{cusp} in this case (namely, the\linebreak periodicity
of $4L(4\widetilde\tau)+6/\pi i\widetilde\tau$) is not satisfied. But on the
rectangle in the statement of Lemma~\ref{cusp}, this function is at least
bounded. Now, if we apply the same reasoning to the holomorphic $1$-form
$L(\tau+1/2)d\tau$, then the boundedness hypothesis of Lemma~\ref{cusp} is
again satisfied, and again periodicity fails. When we subtract
$L(\tau+1/2)d\tau$ from $L(\tau)d\tau$, periodicity is restored in view of
Corollary~\ref{Xi} and boundedness persists! Lemma~\ref{cusp} now applies and
we conclude that $\Xi$ has no worse than a simple pole at $z=\infty$. Similar
reasoning
applies concerning the cusp at $z=1$. With more care we could even compute the
residues at these points (but this is an {\em optional\/} extra).


To conclude, we have verified that 
$$\left(L(\tau)-L(\tau+1/2)\right)d\tau$$
and
$$\big((\theta(\tau))^4-(\theta(\tau+1/2))^4\big)d\tau$$
are meromorphic one-forms on the thrice-punctured sphere with poles and zeros
in the same locations. It follows that one is a constant multiple of the other,
and the proof of (\ref{jacobi})\footnote{The full force of the Jacobi
four-square theorem, namely that the number of ways of representing an integer
$n$ as a sum of four squares of integers is equal to 
$8\sum_{4 \nmid d \mid n}d$, follows from (\ref{jacobi}) in an elementary
fashion.} is complete upon comparing their power series expansions in $q$.

\end{document}